\documentclass{amsart}
\usepackage[english]{babel}
\usepackage[T1]{fontenc}
\usepackage[utf8]{inputenc}
\usepackage{amsmath}
\usepackage{amssymb}
\usepackage{amsfonts}
\usepackage{amsthm}
\usepackage[initials,nobysame]{amsrefs}
\usepackage{bm}
\usepackage{bbm}
\usepackage{mathtools}
\usepackage{ascmac}
\usepackage{framed}
\usepackage{ulem}
\usepackage{graphicx}
\usepackage{colonequals}	
\usepackage{graphicx,color,hyperref,verbatim}
\allowdisplaybreaks

\usepackage{color}

 {\endMakeFramed}

\newcommand{\eps}{\varepsilon}
\newcommand{\C}{\mathbb{C}}
\newcommand{\N}{\mathbb{N}}

\newcommand{\R}{\mathbb{R}}
\renewcommand{\S}{\mathbb{S}}
\newcommand{\T}{\mathbb{T}}

\newcommand{\Z}{\mathbb{Z}}

\newcommand{\boF}{\mathcal{F}}

\newcommand{\boH}{\mathcal{H}}

\newcommand{\boM}{\mathcal{M}}

\newcommand{\boR}{\mathcal{R}}
\newcommand{\boS}{\mathcal{S}}
\newcommand{\boT}{\mathcal{T}}

\newcommand{\ka}{\ensuremath{\kappa}}
\newcommand{\la}{\ensuremath{\lambda}}

\newcommand{\te}{\ensuremath{\theta}}
\newcommand{\al}{\ensuremath{\alpha}}
\newcommand{\be}{\ensuremath{\beta}}
\newcommand{\gam}{\ensuremath{\gamma}}
\newcommand{\Gam}{\ensuremath{\Gamma}}

\newcommand{\si}{\ensuremath{\sigma}}

\newcommand{\Om}{\ensuremath{\Omega}}

\renewcommand{\Re}{\mathop{{\rm Re}}\nolimits}

\newcommand{\loc}{\mathrm{loc}}

\newtheorem*{claim*}{Claim}

\newtheorem*{cor*}{Corollary}
\newtheorem{lem}{Lemma}
\newtheorem{prop}{Proposition}

\newtheorem{thm}{Theorem}

\theoremstyle{definition}

\newtheorem{remark}{Remark}

\theoremstyle{remark}

\makeatletter												%
  \@addtoreset{equation}{section}						%
\makeatother

\newcommand{\nor}[2]{\left\| {#1} \right\|_{#2}}		
\newcommand{\ovl}[1]{\overline{#1}}					
		
\newcommand{\inp}[2]{\langle {#1} , {#2} \rangle }	

\newcommand{\Del}{{\Delta}}							
\newcommand{\del}{{\delta}}								
\newcommand{\rd}{{\partial}}								
\newcommand{\nab}{{\nabla}}							

\newcommand{\wto}{\rightharpoonup}

\newcommand{\bh}{{\mathbf{h}}}
\newcommand{\bn}{{\mathbf{n}}}

\newcommand{\bom}{{\mathbf{m}}}

\newcommand{\bpsi}{{\bm{\psi}}}

\newcommand{\bJ}{{\mathbf{J}}}

\DeclareMathOperator{\BF}{{\rm BF}}

\begin{document}

\title[Time decay estimates for the Landau--Lifshitz--Gilbert equation]{Time decay estimates for localized perturbations around a helical state 
for the Landau--Lifshitz--Gilbert equation}
\author{Ikkei Shimizu}
\address[I.~Shimizu]{Department of Mathematics, Graduate School of Science, Kyoto University, Kyoto, Japan}
\email{shimizu.ikkei.8s@kyoto-u.ac.jp}

\begin{abstract}
We study the dynamics of the Landau--Lifshitz--Gilbert equation with the Dzyaloshinskii--Moriya interaction. 
The equation admits a family of exact stationary solutions, referred to as helical states, which are periodic in one spatial variable and constant in the others. 
We investigate the dynamical stability of a helical state with respect to perturbations belonging to suitable Lebesgue and Sobolev spaces. 
Under a smallness assumption on the initial perturbation, we prove global existence and time decay estimates for solutions, demonstrating that the above helical state is stable. 
The analysis of the relevant linear operator is carried out via the Bloch--Fourier-wave decomposition, 
where the eigenvalue problem for the reduced operator is characterized by certain Mathieu equations. 
\end{abstract}

\maketitle

%
%
%
%
%
%
%
%
%
%
\section{Introduction}\label{S1}
\subsection{Introduction and main result}
We consider the Landau--Lifshitz--Gilbert (LLG) equation: 
\begin{equation}\label{E1.1}
\rd_t \bn = \bn \times E'[\bn] + \al \bn \times (\bn\times E'[\bn]),
\end{equation}
associated with the Landau--Lifshitz energy with Dzyaloshinskii--Moriya interaction (DMI)
\begin{equation}\label{E1.2}
E [\bn] = \frac 12 \int_{\R^d} |\nab \bn|^2 dx + \frac 12 \int_{\R^d} \bn \cdot \nab \times \bn\, dx.
\end{equation}
Here, $d=1,2,3$ is the dimension of the space, and 
$\al>0$ is a constant. 
$\bn = \bn(t,x) = {}^t (n_1(t,x),n_2(t,x), n_3(t,x))\colon [0,\infty)\times \R^d \to \R^3$ is the unknown function of \eqref{E1.1} with constraint 
\begin{equation}\label{E1.3}
|\bn|\equiv 1.
\end{equation}
In other word, the target of the map $\bn$ is the unit sphere $\S^2$ in $\R^3$. 
Also, the curl in \eqref{E1.2} is defined by 
\[
\nab \times \bn
=
\begin{pmatrix}
\rd_1 \\ \rd_2 \\ \rd_3
\end{pmatrix}
\times 
\begin{pmatrix}
n_1 \\ n_2 \\ n_3
\end{pmatrix}
,
\]
where we interpret 
$\rd_2=\rd_3=0$ for $d=1$, and 
$\rd_3=0$ for $d=2$. 
Under \eqref{E1.3}, the variational derivative of $E$ is given by
\[
E'[\bn] = -\Del \bn + \nab \times \bn 
- \left(|\nab \bn|^2 + (\nab \times \bn)\cdot \bn\right) \bn.
\]
Accordingly, \eqref{E1.1} can be explicitly written as
\begin{equation}\label{E1.4}
\rd_t \bn = \bn \times \left(-\Del \bn + \nab \times \bn\right)
+  \al \left(\Del\bn - \nab \times \bn +  
\left(|\nab \bn|^2 + (\nab \times \bn)\cdot \bn\right) \bn\right).
\end{equation}
\eqref{E1.1} is a mathematical 
model to describe the dynamics of 
magnetization vector $\bn$ of ferromagnets with energy-damping effect. 
See \cite{HubertSchafer} for a detailed background; see also 
\cite{MR2771669}*{Section 7} 
for an introductory exposition of the mathematical model. 
The DMI, corresponding to the second term in \eqref{E1.2}, 
arises in certain ferromagnets and is known to 
stabilize 
spatially inhomogeneous configurations which have geometric structures  
\cites{BogYab89, BogHub94, BOGDANOV1999182, NagTok2013}. 
In the present case, \eqref{E1.1} admits a family of explicit stationary solutions
\begin{equation}\label{E1.425}
\bh^\ka (x) =
\begin{pmatrix}
0 \\ \cos \ka x_1 \\ \sin \ka x_1
\end{pmatrix}
,\qquad \ka>0
\end{equation}
which we call the helical states. 
In fact, $\bh^\ka$ solves
\begin{equation}\label{E1.45}
-\Del \bh^\ka + \nab \times \bh^\ka = (\ka^2-\ka) \bh^\ka,
\end{equation}
and hence satisfies $E'[\bh^\ka]=0$ in the classical sense. 
See \cite{Ross2021} for the derivation of these solutions. 
Among this family of solutions, 
$\bh^{1/2}$ can be viewed as a ground state of the energy at formal level. 
In fact, 
the energy density can be rewritten as
\begin{equation}\label{E1.5}
\begin{aligned}
&\frac 12 |\nab \bn|^2 + \frac 12 \bn \cdot \nab \times \bn \\
& =  \frac 12 \left(
\left|\nab \times \bn + \frac 12 \bn\right|^2 
+(\nab \cdot \bn)^2
\right) 
-\frac 18 +\frac 12 \nab\cdot \left( (\bn\cdot \nab)\bn - \bn (\nab \cdot \bn) \right), 
\end{aligned}
\end{equation}
which stems from a well-known formula in the theory of the Oseen--Frank energy (see Section 2.1 in \cite{Stewart} for instance). 
It is known that $\bh^{1/2}$ is the only map, up to rotation, which \textit{minimizes} the bulk terms in \eqref{E1.5} in the following sense:
\begin{align}\label{E1.55}
\nab \times \bh^{1/2} + \frac 12 \bh^{1/2} =0,&&
\nab \cdot \bh^{1/2} =0,&&
 |\bh^{1/2}|\equiv 1.
\end{align}
We refer the reader to Lemmas 1 and 7 in \cite{MR4170329} for the detailed discussion on this fact. 
Moreover, as pointed out in \cite{Ross2021}, 
the minimality at $\ka=\frac 12$ is also evidenced 
by computing energy density of $\bh^{\ka}$:
\[
\frac 12 |\nab \bh^{\ka}|^2 + \frac 12 \bh^{\ka} \cdot \nab \times \bh^{\ka} = 
\frac 12\ka(\ka-1) = \frac 12 \left(\ka-\frac 12\right)^2 -\frac 18.
\]
However, 
these arguments remain at a formal level, 
as the total energy of $\bh^{\ka}$ is not strictly well-defined.  
Therefore, from the viewpoint of PDE theory, 
it is worth investigating rigorously the stability of $\bh^{\ka}$ 
by analyzing the associated dynamical system, such as \eqref{E1.1}.\par
In the present paper, 
we investigate the dynamical stability of $\bh^1$ with respect to \eqref{E1.1}.  
Specifically, 
we study the dynamics of \eqref{E1.1} when the solution is perturbed \textit{locally} from $\bh^1$. 
More precisely, we consider the solution of \eqref{E1.1} of the form
\begin{equation}\label{E1.6}
\bn(t,x) = \bh^1 (x) + \bom (t,x)
\end{equation}
with $\bom(t,\cdot)$ belonging to the inhomogeneous Sobolev spaces $H^s(\R^d)$. 
Here, the term \textit{local} refers to the condition $\bom \in L^2(\R^d)$. 
The purpose of the present paper is to examine the long-time behavior of the solutions to \eqref{E1.1} of the form \eqref{E1.6}. 
In the main theorem, we will claim that 
some norms of $\bom$ decays as $t\to\infty$ under smallness assumptions of initial data for $\bom$. 
This demonstrates the dynamical stability of $\bh^1$ under localize perturbations, 
despite the fact that $\bh^{1}$ is not a ground state in the above meaning.
\par 
In what follows, we simply denote $\bh^1$ by $\bh$. 
To begin with, we substitute \eqref{E1.6} into \eqref{E1.4}, deriving the equation of $\bom$ as
\begin{equation}\label{E1.7}
\rd_t \bom = 
(\al - \bh\times \cdot) (\Del\bom - \nab \times \bom) 
 + \bom\times (-\Del \bom + \nab \times \bom) + \al \Gam (\bom)
\end{equation}
where
\begin{align}\label{E1.75}
\Gam (\bom) = &
\left(
2(\nab \bh : \nab \bom)  + \nab \times \bh \cdot \bom  
+ \nab \times \bom \cdot \bh 
+|\nab \bom|^2  + \nab \times \bom \cdot \bom
\right) (\bh + \bom). 
\end{align}
We also set the initial data $\bom(0,x)=\bom_0 (x)$. 
Note that \eqref{E1.7} is quasi-linear, as it includes the nonlinearity $\bom\times \Del \bom$. 
We first claim that \eqref{E1.7} is locally well-posed in $H^s(\R^d)$, 
the proof of which is given in Appendix \ref{SA}.
\begin{prop}[Local well-posedness]\label{P1} 
Let $s$ be an integer satisfying $s\ge 1$ for $d=1$, and $s\ge 2$ for $d=2,3$. For $\bom_0\in H^s(\R^d\colon \R^3)$, 
\eqref{E1.7} has a unique solution 
\[\bom \in C([0,T_{\max})\colon H^s(\R^d\colon \R^3))\cap\allowbreak 
L^2_{\loc}(0,T_{\max}\colon H^{s+1}(\R^d\colon \R^3))\]
with $\bom(0,x)=\bom_0$, 
where 
$T_{\max} = T_{\max} (\bom_0)$ is the maximal existence time in this class. 
Moreover, 
the solution map $\bom_0 \mapsto \bom$ is continuous in the topology $H^{s-1}$ to $C_tH^{s-1}$. 
Furthermore, if $T_{\max} (\bom_0)<\infty$, then we have
\begin{equation}\label{E1.8}
\lim_{t\to T_{\max}} \nor{\bom(t)}{H^s(\R^d)} = \infty. 
\end{equation}
If $\bom_0$ satisfies $|\bh +\bom_0|=1$, then $|\bh + \bom(t)|=1$ holds for all $t\in [0,T_{\max})$. 
\end{prop}
\begin{remark}\label{R1}
In Proposition \ref{P1}, 
no smallness condition is assumed on the initial data, despite the presence of the second-order derivative nonlinearity $\bom\times \Del \bom$. 
The key to avoid this assumption is to exploit the cancellation structure of $\bom\times \Del \bom$, which prevents the appearance of bad terms in the energy method. 
See the proof of Lemma \ref{L7.14} for the detailed argument. 
This technique is commonly used in the study of LLG, which can also be seen, for example, in \cites{MR1877231,MR2600676}.
\end{remark}
The main theorem of the present paper is the following:
\begin{thm}[Time decay estimates]\label{T1}
Let $s\ge 2$ be an integer. Then there exist $\eps_0>0$ and $C>0$, where 
$C$ is independent of $\al$, 
such that the following holds: 
If $\bom_0\in L^1(\R^d\colon \R^3)\cap H^s(\R^d\colon \R^3)$ satisfies $\nor{\bom_0}{L^1\cap H^s}<\eps_0$ and $|\bh+\bom_0|\equiv1$, then the solution $\bom$ to \eqref{E1.7} in Proposition \ref{P1} is global in time, and satisfies
\begin{equation}\label{E1.9}
\begin{aligned}
(1+\al t)^{\frac d4} \nor{\bom(t)}{H^s (\R^d)}
+(1+\al t)^{\frac d2} \nor{\bom(t)}{W^{\lfloor s-\frac{d+1}{2}\rfloor,\infty}(\R^d)} 
\le C 
\nor{\bom_0}{L^1(\R^d)\cap H^s(\R^d)} 
\end{aligned}
\end{equation}
for all $t\in [0,\infty)$.
\end{thm}
In particular, \eqref{E1.9} gives a pointwise decay estimate:
\[
\left|
\bn (t,x) - \bh (x)
\right| \le C (1+\al t)^{-\frac d2},\quad x\in\R^d,\quad t\ge 0
\]
with some constant $C>0$ independent of $\al$. 
\begin{remark}
By rescaling, we can generalize Theorem \ref{T1} into 
the case of any positive coefficient of DMI term: 
\[
E_K [\bn] = \frac 12 \int_{\R^d} |\nab \bn|^2 dx + \frac {K}2 \int_{\R^d} \bn \cdot \nab \times \bn\, dx,\qquad K>0.
\]
In this case, 
we can obtain the similar stability result for $\bh^{K}$. 
In fact, 
let $\bn$ be a solution to \eqref{E1.1} with the energy replaced by $E_K$. 
Then by change of variables $\tilde\bn (s,y)= \bn (t,x)$ with $s=K^2t$ and $y=Kx$, $\tilde\bn$ solves \eqref{E1.1} with the energy \eqref{E1.2}. 
\end{remark}
\begin{remark}
The specific choice $\ka=1$ in the present paper is due to the following technical reason. 
In our analysis, we reduce the equation of $\bom$ to a new $\C$-valued PDE 
by using a moving frame. Then the linearized operator turns out to be $\C$-linear only when $\ka=1$, which will be discussed in Remark \ref{Rxx5}. 
Our linear analysis essentially relies on this property.
\end{remark}
\subsection{Outline and the idea of the proof}
We now outline the proof of Theorem \ref{T1} and discuss the underlying ideas. 
We first introduce a moving frame of 
the tangent space of $\S^2$ at $\bh$, 
and transform $\bom$ into a new $\C$-valued unknown $u$. 
Then we derive the dynamical equation of $u$\ (see Section \ref{S2}), 
which turns out to 
be a semilinear partial differential equation of the form
\begin{align}\label{E1.12}
\rd_t u = (-\al+i) Au + o(u),
&&
A= 
\left\{
\begin{aligned}
& -\rd_{1}^2, \quad&\text{if } d=1,\\
&-\Del +i\cos x_1 \rd_2, \quad &\text{if } d=2,\\
&-\Del + i \cos x_1 \rd_2 + i\sin x_1 \rd_3, \quad &\text{if } d=3,
\end{aligned}
\right.
\end{align}
where $o(u)$ is the collection of the super-linear terms with respect to $u$ (see \eqref{E2.4} and \eqref{E2.5}). 
Since $u$ is quantitatively equivalent to $\bom$ (see \eqref{E2.275} and \eqref{E2.3}), 
we are able to reduce our problem to the examination for \eqref{E1.12}. \par
The analysis of \eqref{E1.12} consists of three parts. 
In the first step, we investigate the spectral properties of the linear operator $A$ (see Section \ref{S3}). 
The main focus is the case $d=2,3$, while the case $d=1$ is trivial. 
The characteristic feature of $A$ is that the coefficient depends 
periodically on $x_1$, while being constant in the other variables. 
Shedding light on this, we are led to apply the Bloch transform in $x_1$-variable, and the Fourier transform in the others. 
Then $A$ is decomposed into the direct sum of the operators $\{A_\xi\}$ on $L^2(\R/2\pi\Z)$ indexed by the wave number $\xi=(\xi_1,\xi')\in (\R/\Z)\times \R^{d-1}$. 
It will be turned out that the eigenvalue problem for $A_\xi$ is equivalent to that for the Mathieu operators. 
By the energy-band theory, combining with the estimate using the min-max principle, we identify that the least eigenvalue $\la_0(\xi)$ of $A_\xi$ with $\xi$ in the neighborhood of $0$ is responsible for the low-frequency part of $A_\xi$ (see Proposition \ref{P2}). 
Then by using the Lyapunov--Schmidt reduction, we derive the asymptotics of 
$\la_0(\xi)$ 
and its associated eigenfunction $\phi_0(\xi,\cdot)$ 
around $\xi=0$ (see Proposition \ref{P3}).\par 
In the second step, we consider the associated linear problem
\[
\rd_t u = (-\al +i) A u + F,\qquad F=F(t,x)\colon [0,\infty) \times \R^d \to \C. 
\]
Following the framework of \cites{MR2601067, MR2788923}
, 
we obtain the linear estimates separately for high- and low- frequency parts (see Section \ref{S4}). 
In detail, 
for the high-frequency part, 
we derive the energy estimate which is analogous to the standard one for the heat equation (see Proposition \ref{P:4}). 
On the other hand, for the low-frequency part, 
we prove $L^1\to L^2$ and $L^1\to L^\infty$-type estimates 
by using the before-mentioned asymptotics of $\la_0(\xi)$ and $\phi_0(\xi,\cdot)$ (see Proposition \ref{P7.20}). 
We note that the decay rates in these estimates are the same as those for the heat equation, 
thanks to the fact that the asymptotics of $\la_0(\xi)$ and $\phi_0(\xi,\cdot)$ on $\xi$ 
coincide for up to the first order derivatives.\par
In the final part, we conclude \eqref{E1.9} applying the linear estimates to \eqref{E1.12} (see Section \ref{S5}). 
Following the framework of \cites{MR3243362,MR4316931}, we estimate the solutions by considering \eqref{E1.12} as a system of low- and high- frequency parts. 
The resulting decay rate is the same as that of the linear estimate for the low-frequency part. 
\par
\subsection{Related works}
We briefly mention some related mathematical works. 
The variational problem for the Landau--Lifshitz energy with DMI 
has been extensively investigated, see  \cites{MR3269033,MR3639614,MR3853081,MR4091507,MR4179069,MR4630481,ibrahim2025global,MR4198719, MR4961803} and references cited therein. 
In contrast, the associated dynamical problem has recently begun to be examined. 
In \cite{MR4482034}, 
the author addressed the Landau-Lifshitz equation without the Gilbert damping term (Schr\"odinger map case),  establishing the local well-posedness in the Sobolev class. 
D\"oring and Melcher \cite{MR3639614} considered the LLG equation including spin-current terms and investigated the conformal limit of solutions with certain energy bound. 
Furthermore, \cite{MR4616656} examined the static domain-wall solutions to 1D problem and proved their stability with respect to LLG equation. 
We note that when the energy consists only of the Dirichlet energy, 
the corresponding dynamical equation has been intensively investigated in the context of the harmonic map heat flow or Schr\"odinger maps. 
See, for example, \cites{MR2725187,MR2800718,bejenaru2024nearsolitonevolution2equivariant,MR4932382} and references cited therein. 
On the other hand, the Bloch-wave analysis has been applied to 
the study of stability for spatially periodic equilibriums 
across various kind of nonlinear PDEs. 
We refer to the seminal pioneering works \cites{MR1395210,MR1434157,MR1617491,MR1657387,MR1482940} on this topic, 
as well as the study of periodic waves for water-wave equations \cites{MR2215999,MR2601067,MR2788923} and 
dispersive partial differential equations \cites{MR2791492,MR3656518,MR3749122}. 
Similar techniques used in these works also appear in the study of compressible Navier--Stokes equation in some periodic settings (see, for instance,  \cites{MR3348113,MR3606303,MR3714502}). 
We also refer to the literatures concerning periodic eigenvalue problems \cites{MR493421,MR3075381,MR2978285}.\par
\subsection{Notations}
We conclude the introduction by giving notations appearing in this paper, most of which are based on the standard convention. 
For $f\colon\R^d\to \C$, let $\boF f\colon\R^d\to\C$ be the Fourier transform of $f$, explicitly given by $\boF f (\xi) = \int_{\R^d} e^{-i\xi\cdot x} f(x) dx$. 
For two maps $\bom = {}^t(m_1,m_2,m_3)$, $\bn={}^t(n_1,n_2,n_3)\colon \R^d\to \R^3$, we write 
$\nab \bom: \nab \bn = 
\sum_{j=1,...,d}\sum_{k=1,2,3} \rd_{j} m_k \rd_ j n_k$. 
For $A,B\ge 0$, 
we write $A\sim B$ if $C^{-1} A\le B\le CA$, and $A\lesssim B$ if $A\le CB$, where 
$C>0$ is a constant independent of the parameters involved. 
The implicit constant $C>0$ in this paper will vary from line by line. 
For a set $\Om$, we denote the characteristic function with respect to $\Om$ by $\mathbbm{1}_\Om$. For a multi-index $\gam = (\gam_1,...,\gam_d)\in (\Z_{\ge 0})^d$, we set $|\gam| = \gam_1+...+\gam_d$. 
The Lebesgue space on a measure space $X$ with index $p\in [1,\infty]$ is denoted by $L^p(X)$, along with the associated Lebesgue norm $\nor{f}{L^p(X)}$. 
Letting $X$ be either the Euclidean space $\R^d$ or torus $\T^d$ for $d\in \N$, 
we write $W^{s,p}(X)$ to represent 
the inhomogeneous Sobolev space with Sobolev index $s\in\Z_{\ge 0}$ and Lebesgue index $p\in [1,\infty]$, the equipped norm of which is defined as $\nor{f}{W^{s,p}(X)} = \left[\sum_{|\al|\le s} \nor{\rd^\al f}{L^p(X)}^2 \right]^{1/2}$. 
If $p=2$, we denote $W^{s,2}(X) = H^s(X)$. 
We often abbreviate $L^p(X)$, $W^{s,p}(X)$, $H^{s}(X)$ as $L^p$, $W^{s,p}$, $H^s$ if it is clear from the context. 
For $\R^3$-valued function spaces, we write such as $L^p(X\colon\R^3)$, 
or simply $L^p(X)$ if no ambiguity arises. 
For a Hilbert space $X$, the associated inner product is denoted by $\inp{\cdot}{\cdot}_X$, or simply $\inp{\cdot}{\cdot}$ if there is no risk of confusion. 
In particular, 
the $L^2$-inner product is defined as $\inp{f}{g}_{L^2(X)} = \int_{X} f(x) \ovl{g(x)} dx$ for $f,g\in L^2(X)$, and 
$\inp{\bm{F}}{\bm{G}}_{L^2(X\colon \R^3)} = \int_{X} \bm{F}(x) :\bm{G}(x) dx$ for $\bm{F},\bm{G}\in L^2(X\colon \R^3)$. 
For measure spaces $X,Y$, numbers $p\in [1,\infty]$, $q\in [1,\infty]$, and a measurable function $f=f(x,y)\colon X\times Y\to\C$, 
we obey the convention that $\nor{f}{L^p_{x}L^q_{y}} = \left[ 
\int_X \nor{f(x,\cdot)}{L^q(Y)}^p dx\right]^{1/p}$ if $p<\infty$, whereas we replace by $\rm{esssup}_{x\in X}$ when $p=\infty$. 
Here, let $f(t,x)$ be a function of time $t$ and space $x$, 
and let $X$ be a normed space for functions in $x$. 
For any interval $(a,b)\subset\R$ and $p\in [1,\infty]$, 
we denote $\nor{f}{L^p(a,b\colon X)} = 
\left[ \int_a^b \nor{f(t,\cdot)}{X}^p dt\right]^{1/p}$, where the case $p=\infty$ is interpreted similarly. 
Specifically when $a=0$ and $b=T>0$, we simply write $\nor{f}{L^p_TX_x}= \nor{f}{L^p(0,T\colon X)}$. 
Finally, for $0<T\le \infty$, $f\in L^p_{\loc} (0,T\colon X)$ means that $\nor{f}{L^p(0,a\colon X)}<\infty$ for all $0<a<T$. 
%
%
%
%
%
%
%
%
%
\section{Reduction via moving frame}\label{S2}
In this section, we introduce a transformation which maps $\bom$ into a new $\C$-valued unknown function $u$ by using a moving frame. 
Then we derive the equation of $u$ 
from \eqref{E1.7}. 
The motivation behind this reduction arises 
from the fact that 
the perturbation $\bom$ obeys a constraint due to \eqref{E1.3}. 
However, 
when $|\bom|$ is sufficiently small, 
$\bom$ can be uniquely corresponded to 
a tangent vector of $\S^2$ at $\bh$. 
Then we associate with its coefficients of a fixed orthonormal frame of the tangent space. 
This framework is commonly referred to as the moving frame method, 
which is an effective tool to remove geometric constraints (see, for instance,  \cites{MR2233925,MR1737504,MR2725187,MR2462113}). In particular, our argument is similar to that in \cite{MR2462113}.\par 
To begin with, let us consider a time-independent map $\bn\colon \R^d\to\S^2$ with 
$\bn= \bh+ \bom$. 
Set
\begin{equation}\label{E2.05}
\bJ_1 =
\begin{pmatrix}
0 \\ -\sin x_1 \\ \cos x_1
\end{pmatrix}
,
\qquad
\bJ_2 =
\begin{pmatrix}
1 \\ 0 \\ 0
\end{pmatrix}
.
\end{equation}
Notice that $\{\bJ_1, \bJ_2, \bh \}$ forms an orthonormal frame of $\R^3$.  
Hence we can uniquely write
\begin{equation}\label{E2.1}
\bom = u_1 \bJ_1 + u_2 \bJ_2 + z \bh
\end{equation}
with $u_1,u_2,z \colon \R^d\to\R$. 
Then we set $u=u_1+iu_2$, which is a map from $\R^d$ to $\C$. 
Now we observe that the correspondence $\bom\mapsto u$ 
is one-to-one when the amplitudes are sufficiently small. 
Indeed, restricting to $|\bom|<1$, we have 
$|z|\le 
|\bom|<1$, and thus 
$z$ is uniquely determined by $u$ via
\begin{equation}\label{E2.2}
1 =|\bn|^2 = |u|^2 + (1+z)^2 \qquad \Longleftrightarrow \qquad
z = -1 +\sqrt{1 - |u|^2}.
\end{equation}
Throughout the paper, we consistently apply this transform for maps satisfying $|u|\le \frac 12$ everywhere. 
In such case, it follows from \eqref{E2.2} that 
\begin{equation}\label{E2.25}
|z| \lesssim |u|^2 ,\qquad |\nab z| \lesssim |u||\nab u|,
\end{equation}
which, in particular, implies 
\begin{equation}\label{E2.275}
|\bom|\sim |u|.
\end{equation}
Moreover, 
for $s\in \Z_{\ge 0}$ and $p\in [\frac d2, \infty]$ with $p>1$, it holds that
\begin{align}\label{E2.3}
\nor{\bom}{W^{s,p}(\R^d)} &\le C_{s,p}
\left(\nor{u}{W^{s,p}(\R^d)} + 
\nor{u}{W^{s,p}(\R^d)}^{s}\right), \\
\label{e2:2}
\nor{u}{W^{s,p}(\R^d)} &\le C_{s,p}  \nor{\bom}{W^{s,p}(\R^d)}.
\end{align}
For completeness, we will provide a proof of \eqref{E2.3} in Appendix \ref{SB}. 
Next, we derive the dynamical equation for $u$. 
The following computation is conducted in the case $d=3$, 
whereas 
the results in the case $d=1$, $d=2$ are obtained just by letting $\rd_2=\rd_3=0$, $\rd_3=0$, respectively. 
Let 
$\bom \in C([0,T] \colon H^s(\R^d))\cap L^2 (0,T\colon H^{s+1}(\R^d))$ 
be a solution to \eqref{E1.7} 
with $s\ge 2$ satisfying $\nor{u}{L^\infty_T L^\infty_x} \le \frac 12$. 
Let $u(t,\cdot)$ be the associated function in \eqref{E2.1} at each $t\in [0,T]$.  
Then, we have 
\[
\rd_t u = \rd_t \bom \cdot (\bJ_1+ i\bJ_2).
\]
Now we substitute \eqref{E1.7} and compute the right hand side. 
First note that for $j=2,3$, 
\begin{align}\label{Ea2.6}
\rd_1 \bJ_1 = -\bh,&&
\rd_1 \bh = \bJ_1,&& 
\rd_1 \bJ_2 = \rd_j \bJ_1 = \rd_j \bJ_2 = \rd_j \bh = 0.
\end{align}
By \eqref{Ea2.6}, direct computation yields
\begin{align*}
\Del \bom &=
(\Del u_1 - u_1 + 2 \rd_1 z) \bJ_1
+ \Del u_2 \bJ_2
+ (\Del z- 2\rd_1 u_1 -z) \bh,
\\
-\nab \times \bom
&=
(u_1+\rd_2 u_2 \cos x_1 + \rd_3 u_2 \sin x_1 - \rd_1 z)\bJ_1\\
&\qquad + (-\rd_2 u_1 \cos x_1 - \rd_3 u_1 \sin x_1 - \rd_2 z \sin x_1 + \rd_3 z \cos x_1) \bJ_2 \\
&\qquad 
+ (\rd_1 u_1 + \rd_2 u_2 \sin x_1 - \rd_3 u_2 \cos x_1 +z) \bh.
\end{align*}
Hence 
\begin{align*}
&\left[(\al - \bh\times \cdot ) (\Del \bom -\nab\times \bom)\right] 
\cdot (\bJ_1 + i\bJ_2)\\
&=
(\al-i) \left(
\Del u -i \rd_2 u \cos x_1 - i\rd_3 u \sin x_1 + \rd_1 z - i\rd_2 z \sin x_1
+ i \rd_3 z \cos x_1
\right),
\\
&\left[\bom\times (\Del \bom -\nab\times \bom)\right] 
\cdot (\bJ_1 + i\bJ_2)\\
&=
u_2 (\Del z -\rd_1 u_1 +\rd_2 u_2 \sin x_1 -\rd_3u_2\cos x_1) 
\\
&\qquad 
- z (\Del u_2 - \rd_2 u_1 \cos x_1 -\rd_3 u_1 \sin x_1 -\rd_2 z \sin x_1 +\rd_3 z \cos x_1) 
\\
&\qquad 
+i 
\left[
z (\Del u_1 +\rd_1 z +\rd_2 u_2 \cos x_1 + \rd_3 u_2 \sin x_1) \right.\\
&\hspace{40pt}\left.
- u_1 (\Del z - \rd_1 u_1 +\rd_2 u_2 \sin x_1 - \rd_3 u_2 \cos x_1)
\right] \\
&=
i(-\Del z +\rd_1 u_1 - \rd_2 u_2 \sin x_1 + \rd_3 u_2 \cos x_1 ) u \\
&\qquad 
+z (i\Del u + \rd_2 u \cos x_1 + \rd_3 u \sin x_1
+\rd_2 z \sin x_1 - \rd_3 z \cos x_1 + i\rd_1 z
).
\end{align*}
For each term of $\Gam(\bom)$, we have
\begin{align*}
2(\nab\bh: \nab \bom) 
&=
2(\rd_1 u_1 + z) ,
\\
\nab\times \bh \cdot \bom 
&=
-z ,
\\
\nab \times \bom \cdot \bh
&=
- \rd_1 u_1 -\rd_2 u_2 \sin x_1 + \rd_3 u_2 \cos x_1 -z
,\\
|\nab \bom|^2  
&=
|\nab u|^2 + |\nab z|^2 + u_1^2+z^2
+ 2z\rd_1 u_1 - 2 u_1 \rd_1 z,
\\
\nab \times \bom \cdot \bom 
&=
u_1 (-u_1 -\rd_2 u_2 \cos x_1 - \rd_3 u_2 \sin x_1 +\rd_1 z) \\
&\qquad + u_2 (\rd_2 u_1 \cos x_1 +\rd_3 u_1 \sin x_1 +\rd_2 z \sin x_1 - \rd_3 z \cos x_1) \\
&\qquad 
+z (-\rd_1 u_1 -\rd_2 u_2 \sin x_1 + \rd_3 u_2 \cos x_1 - z)
.
\end{align*}
Hence
\begin{align*}
&\Gam(\bom) \cdot (\bJ_1+i\bJ_2) \\
&= 
\left[
\rd_1 u_1 - \rd_2 u_2 \sin x_1 + \rd_3 u_2 \cos x_1
+ |\nab u|^2+ |\nab z|^2
 \right.\\
&\qquad + (u_2 \rd_2 u_1 - u_1 \rd_2 u_2) \cos x_1 
+(u_2 \rd_3 u_1 - u_1 \rd_3 u_2) \sin x_1 \\
&\qquad 
\left.
+ z\rd_1 u_1 - u_1 \rd_1 z
+ (u_2 \rd_2 z - z\rd_2 u_2) \sin x_1
+ (z \rd_3 u_2 - u_2\rd_3 z) \cos x_1
\right] u.
\end{align*}
%
%
Therefore, we obtain the equation of $u$ as
\begin{equation}\label{E2.4}
\rd_t u = (-\al + i)A u + \al N_1(u,\nab u) + N_2 (u,\nab u ,\nab^2 u),
\end{equation}
where
\begin{equation}\label{E2.5}
A = 
\left\{
\begin{aligned}
& -\rd_{1}^2 &\text{if } d=1,\\
&-\Del +i\cos x_1 \rd_2, \quad &\text{if } d=2,\\
&-\Del + i \cos x_1 \rd_2 + i\sin x_1 \rd_3, \quad &\text{if } d=3,
\end{aligned}
\right.
\end{equation}
and
\begin{align*}
&N_1(u,\nab u) \\
&= 
\rd_1 z - i\rd_2 z \sin x_1 + i\rd_3 z \cos x_1
\\
&\qquad 
+\left[ 
\rd_1 u_1 - \rd_2 u_2 \sin x_1 + \rd_3 u_2 \cos x_1
+ |\nab u|^2+ |\nab z|^2
\right.
\\
&\hspace{40pt} 
+ (u_2 \rd_2 u_1 - u_1 \rd_2 u_2) \cos x_1 
+(u_2 \rd_3 u_1 - u_1 \rd_3 u_2) \sin x_1
\\
&\hspace{40pt} 
\left.
+ z\rd_1 u_1 - u_1 \rd_1 z
+ (u_2 \rd_2 z - z\rd_2 u_2) \sin x_1
+ (z \rd_3 u_2 - u_2\rd_3 z) \cos x_1
\right] u,
\\
&N_2(u,\nab u,\nab^2 u) \\
&= 
-i\rd_1 z -\rd_2 z \sin x_1 + \rd_3 z \cos x_1
\\
&\qquad
+iu\left(
\Del z - \rd_1 u_1 + \rd_2 u_2 \sin x_1 - \rd_3 u_2 \cos x_1
\right) 
\\
&\qquad 
-z \left(
i\Del u + \rd_2 u \cos x_1 + \rd_3 u \sin x_1 + \rd_2 z \sin x_1 
- \rd_3 z \cos x_1 + i \rd_1 z
\right).
\end{align*}
\begin{remark}\label{R4}
The operator $A$ also arises when we consider the formal linearization of the energy around $\bh$. 
In detail, let $\bom\in C_0^\infty (\R^d\colon \R^3)$ be a function with $|\bh+\bom|\equiv 1$. 
Then, the difference of energy density between $\bh+\bom$ and $\bh$ is rewritten as
\begin{align*}
&\frac 12 |\nab (\bh+ \bom)|^2 + \frac 12 (\bh+\bom) \cdot \nab\times (\bh+\bom)
-\frac 12|\nab \bh|^2 - \frac 12 \bh\cdot \nab \times \bh \\
&= 
\nab \bh : \nab \bom + \frac 12 \bom \cdot \nab \times \bh + \frac 12 \bh \cdot \nab \times \bom 
+ \frac 12 |\nab \bom|^2 + \frac 12 \bom \cdot \nab \times \bom.
\end{align*}
Integrating the right hand side over $\R^d$, we obtain
\begin{align*}
&\int_{\R^d} 
\left(
\nab \bh : \nab \bom + \frac 12 \bom \cdot \nab \times \bh + \frac 12 \bh \cdot \nab \times \bom 
+ \frac 12 |\nab \bom|^2 + \frac 12 \bom \cdot \nab \times \bom\right) dx \\
&=
\int_{\R^d} 
\left(
-\Del \bh \cdot \bom +  \bom \cdot \nab \times \bh 
+ \frac 12 |\nab \bom|^2 + \frac 12 \bom \cdot \nab \times \bom\right) dx\\
&= 
\frac 12 \int_{\R^d} 
\left(|\nab \bom|^2 + \bom\cdot \nab\times \bom\right) dx \eqqcolon \boH[\bom],
\end{align*}
where we applied integration by parts and \eqref{E1.45}. 
$\boH$ can be viewed as the Hessian of energy at $\bh$ in the formal sense. 
Now we write $\bom$ in \eqref{E2.1}, and suppose that $|u|\le \frac 12$. 
Then, by \eqref{E2.2}, 
we have
\begin{equation}\label{E2.7}
\boH[\bom] = \int_{\R^d} \Re (Au) \ovl{u} dx + o \left(\nor{u}{H^1(\R^d)}^2\right).
\end{equation}
\end{remark}
\begin{remark}\label{Rxx5}
We briefly remark on the case for general parameter $\ka>0$ in \eqref{E1.425}. 
For the solution $\bn$ to \eqref{E1.4}, 
we consider the perturbation $\bom^\ka = \bn-\bh^\ka$. 
Substituting this into \eqref{E1.4} yields the evolution equation for $\bom^\ka$ as
\[
\rd_t \bom^\ka = (\al - \bh^\ka\times \cdot) (\Del \bom^\ka -\nab \times \bom^\ka + (\ka^2-\ka) \bom^\ka) + \text{higher order terms}.
\]
Here, the moving frame analogous to \eqref{E2.05} is $\bJ^{\ka}_1={}^t(0,\cos \ka x_1, \sin \ka x_1)$ and 
$\bJ^{\ka}_2 = {}^t(1,0,0)$. 
If we set $\bom^\ka= u_1 \bJ_1^\ka + u_2 \bJ^\ka_2 + z \bh^\ka$ and $u=u_1+iu_2$, the equation for $u$ becomes
\begin{align*}
\rd_t u = L_\la u + o(u),
\end{align*}
where
\begin{align*}
L_\ka = (-\al+i) (-\Del u + i \cos \ka x_1 \rd_2 u
+i \sin \ka x_1 \rd_3 u
 - i (\ka^2-\ka)u_2).
\end{align*}
Note that $L_\ka$ is $\C$-linear if and only if $\ka=1$. 
\end{remark}
%
%
%
%
%
%
%
%
%
%
%
\section{Spectral analysis}\label{S3}
\subsection{Bloch--Floquet decomposition}
In this section, we investigate the spectral properties of $A$ defined in \eqref{E2.5} for $d=2,3$. 
In our analysis, we employ the Bloch--Floquet theory for the periodic differential operators. 
To begin with, we recall some standard results from the general theory needed for the following argument. 
For more about the theory, we refer the interested reader to 
Section XIII.16 in \cite{MR493421}, 
Chapter 1 in \cite{MR2978285}, and \cite{MR3075381}. 
See also Appendix A in \cite{MR1482940} 
for detailed rigorous verification of the formulae \eqref{E3.55} and \eqref{E3.8}. \par
We set $\T = \R/(2\pi \Z)$ and $\T^* = (\R/\Z) \times \R^{d-1}$. 
In our argument, 
we identify $\R/(2\pi\Z)$, $\R/\Z$ with $[0,2\pi)$, $[0,1)$ respectively. 
For $\xi_1 \in \R/\Z$, we define
\[
|\xi_1|_* = \min_{n\in\Z} |\xi_1 -n| = \min\{\xi_1 , 1-\xi_1 \},
\]
which is the distance between $\xi_1$ and $0$ in $\R/\Z$. 
We also write $|\xi|_* = (|\xi_1|_*^2 + (\xi')^2)^{1/2}$ for $\xi= (\xi_1,\xi')\in \T^*$. 
For $f\in L^2(\R^d)$, we define $\BF[f] \colon \T^* \times \T \to \C$ as
\begin{equation}
  \BF [f] (\xi, x_1) = \frac 1{(2\pi)^d} \sum_{k\in \Z} e^{ikx_1} \boF [f] (\xi_1 + k , \xi'),
\end{equation}
which is the combination of the Bloch transform in $x_1$ and the Fourier transform in the other variables. 
This is referred to as the Bloch--Fourier transform in \cites{MR2601067, MR2788923}. 
It is known that $\BF$ is an isometry from $L^2 (\R^d)$ to $L^2 (\T^* \times \T)$ up to constant multiple, namely, 
\begin{equation}\label{E3.3}
\nor{\BF [u]}{L^2 (\T^*\times \T)}^2 = \frac{1}{(2\pi)^{d-1}} \nor{u}{L^2(\R^d)}^2.
\end{equation}
For $U=U(\xi ,x_1) \colon \T^* \times \T\to \C$, 
the inversion $\BF^{-1}$ is explicitly represented as
\begin{equation}
  \BF^{-1} [U] (x) = 
\int_{\T^*} e^{i\xi\cdot x} U(\xi,x_1) d\xi
= \int_{0}^{1} \int_{\R^{d-1}} e^{i \xi\cdot x} U(\xi,x_1) d\xi'd\xi_1, 
\end{equation}
where $U(\xi,\cdot)$ in the integrand is interpreted as a $2\pi$-periodic function on $\R$, naturally extended from $[0,2\pi)$. 
The corresponding inversion formula
\begin{equation}\label{E3.5}
  u(x) =
  \int_{\T^*} e^{i\xi\cdot x} \BF [u] (\xi , x_1) d\xi
\end{equation}
is referred to as the direct integral (see \cite{MR493421}). \par
Now, let $A$ be the operator as in \eqref{E2.5}. 
By the Kato--Rellich theorem (see \cite{MR1335452}*{Theorem V.4.3}), 
it is straightforward to see that $A$ is a self-adjoint operator on $L^2(\R^d)$ 
with $D(A)= H^2(\R^d)$.  
Through \eqref{E3.5}, $A$ is represented as
\begin{equation}\label{E3.55}
A u(x) =
\int_{\T^*} e^{i\xi\cdot x} A_\xi \BF [u] (\xi , x_1) d\xi,
\end{equation}
where $A_\xi$ is the self-adjoint operator on $L^2(\T)$ with $D(A_\xi)=H^2(\T)$, defined by
\begin{equation}
\begin{aligned}
A_\xi =  e^{-i\xi\cdot x} A e^{i\xi\cdot x}
=
\begin{cases}
-\rd_{x_1}^2 -2i \xi_1 \rd_{x_1} +|\xi|^2 - \xi_2 \cos x_1,  & \text{if } d=2,\\
-\rd_{x_1}^2 -2i \xi_1 \rd_{x_1} + |\xi|^2 - \xi_2 \cos x_1 
- \xi_3 \sin x_1,  & \text{if } d=3
\end{cases}
\end{aligned}
\end{equation}
for $\xi=(\xi_1,\xi')\in \T^*$. 
Regarding the spectrum of $A$, 
the following identity holds, known as the \textit{Bloch--Floquet decomposition}:
\begin{equation}\label{E3.8}
  \si (A) =
  \ovl{
    \bigcup_{\xi \in \T^*} \si(A_\xi)
  }.
\end{equation}
Thanks to \eqref{E3.8}, the spectral analysis for $A$ can be reduced to that for $A_\xi$.
\subsection{Eigenvalue problem for $A_\xi$}
Let $\xi \in \T^*$. 
We first note that 
the eigenvalue problem $A_{\xi_1,\xi'} f = \la f$ on $\T$ is equivalent to
\begin{equation}\label{E3.9}
A_{0,\xi'} g 
\end{equation}
\begin{equation}\label{E3.10}
g=g(x_1)\colon [0,2\pi]\to\C,\quad
g(0)= e^{-2\pi i\xi_1} g(2\pi),\quad g'(0) =e^{-2\pi i \xi_1 }  g' (2\pi).
\end{equation}
Indeed, these are associated via $g(x_1) =e^{i\xi_1 x_1} f(x_1)$. 
Here, \eqref{E3.9} can be written as
\[
\rd_{x_1}^2g + \left(\la-\xi_2^2 + \xi_2 \cos x_1\right)g  = 0
\]
in the case $d=2$, and 
\[
\rd_{x_1}^2g + \left(\la-\xi_2^2-\xi_3^2 + \xi_2 \cos x_1 + \xi_3 \sin x_1\right)g = 0
\]
in the case $d=3$. 
These coincide with the Mathieu equation, a special class of Hill's equation, 
for which the general theory for the periodic eigenvalue problem 
is well established. 
We recall here some of the know results.  
For the details of these facts, we refer the reader to 
\cite{MR493421}*{Section XIII.16}, \cite{MR2978285}*{Chapter 1}, \cite{MR3075381} for example. 
First, the spectrum of $A_\xi$, $\si(A_\xi)$, consists only of discrete spectrum: 
$\si(A_{\xi}) = \{\la_n(\xi) \}_{n=0}^\infty\subset \R$ 
with $\la_n(\xi)$ strictly increasing in $n$. 
$\la_n(\xi)$ is continuous in $\xi_1\in \R/\Z$, and every spectrum is simple (see 
Section XIII.16, Example 1 in \cite{MR493421} or 
Theorem 1.7.1 in \cite{MR2978285}). 
Accordingly, we set the corresponding eigenfunction by $\phi_n(\xi,\cdot)$ with 
$\nor{\phi_n(\xi,\cdot)}{L^2(\T)}=1$. 
Here, there is ambiguity of the phase in the definition of $\phi_n$, 
but whichever choice will work in the subsequent argument except $\phi_0$; 
we will later fix a specific phase for $\phi_0$ near $\xi=0$ in Proposition \ref{P3}. 
Concerning the dependence of $\la_n(\xi)$ on $\xi_1$, the following monotonicity is known:
\begin{lem}[\cites{MR493421,MR3075381}]\label{L1}
For every $\xi'\in\R^{d-1}$, the map $\xi_1 \mapsto \la_n (\xi_1,\xi')$ is monotonically increasing on 
$[0,\frac 12]$ if $n$ is even, while monotonically decreasing on $[0,\frac 12]$ 
if $n$ is odd. 
Moreover, we have
\[
\la_0 (0,\xi') < \la_0 (\frac 12, \xi') < \la_1 (\frac 12, \xi')
< \la_1 (0,\xi') < \la_2 (0, \xi') < \la_2 (\frac 12 ,\xi') < \cdots
\]
with $\la_n(0,\xi'), \la_n (\frac 12 ,\xi') \to \infty$ as $n\to\infty$ for all $\xi'\in\R^{d-1}$.
\end{lem}
Now we turn to investigate further properties of $\la_n(\xi)$ and $\phi_n(\xi,\cdot)$. 
In the subsequent subsections, 
we apply two different methods to examine the spectrum. 
First, we apply the min-max principle to obtain lower bounds for $\la_n (\xi)$, which hold true globally in $\xi$ (see Proposition \ref{P2}). 
Second, we carry out the Lyapunov--Schmidt reduction to 
derive more precise asymptotics of $\la_0(\xi)$ and $\phi_0(\xi,x_1)$ near $\xi=0$ (see Proposition \ref{P3}).\par
Prior to the main analysis, we provide two preliminary observations. 
First, in the case $d=3$, $A_\xi$ can be rewritten as
\begin{gather*}
A_\xi = -\rd_{x_1}^2 - 2i\xi_1 \rd_{x_1} +\xi_1^2+ |\xi'|^2 -|\xi'| \cos(x_1 - \al_{\xi'}),\\
    \cos \al_{\xi'} = \frac{\xi_2}{|\xi'|},\quad  \sin \al_{\xi'} = \frac{\xi_3}{|\xi'|},
\end{gather*}
which is equivalent to
\begin{align}\label{E3.105}
\boT_{\al_{\xi'}}^{-1} A_\xi \boT_{\al_{\xi'}} =  -\rd_{x_1}^2 - 2i\xi_1 \rd_{x_1} +\xi_1^2 + |\xi'|^2 -|\xi'| \cos x_1,
\end{align}
where $\boT_c f= f(\cdot -c)$ for $f\in L^2(\T)$, $c\in\R$. 
This allows us to reduce the spectral analysis in the 3D case to that in the 2D case. More precisely, $\la_n(\xi_1,\xi')$ in the 3D case is equal to $\la_n(\xi_1,|\xi'|)$ in the 2D case, and the corresponding eigenfunction is simply related to $\boT_{\al_{\xi'}} \phi_n(\xi_1,|\xi'|,\cdot)$.\par
The second observation is the symmetry of $\la_n(\xi)$ on $\xi$.
\begin{lem}
For all $\xi=(\xi_1,\xi')\in \T^*$ and $n\in \Z_{\ge 0}$, we have
\begin{equation}\label{E3.11}
\la_n(\xi_1,\xi') = \la_n (\xi_1,-\xi'),\qquad \text{and}\qquad
\la_n(\xi_1,\xi') = \la_n (1-\xi_1,\xi').
\end{equation}
\end{lem}
\begin{proof}
The first identity follows from
  \[
    A_{\xi_1,-\xi'} = \boT_{\pi}^{-1} A_{\xi} \boT_\pi ,
  \]
while the latter is obtained from the formula
  \[
    A_{1-\xi_1,\xi'} = e^{-ix_1} \boR A_\xi \boR e^{ix_1},
  \]
where $\boR f = \ovl{f}$. 
\end{proof}
\subsection{Lower bound of the spectrum}
We first claim the following estimates for $\la_n(\xi)$.
\begin{prop}\label{P2}
There exist $\te_0>0$ and $C>0$ such that for any $\xi\in \T^*$, we have
\begin{gather}
\label{E3.12}
\la_0(\xi) \ge \te_0 |\xi|_*^2.\\
\label{Ea3.13}
\la_n (\xi) \ge C,\quad n\ge 1.
\end{gather}
\end{prop}
\begin{remark}
Proposition \ref{P2}, combined with \eqref{E3.8}, shows that 
$\si(A)\subset [0,\infty)$, which guarantees the stability of the corresponding linear flow in the spectral sense. 
Also from the viewpoint of \eqref{E2.7},  
this implies that the leading term of the formal Hessian of the energy at $\bh$ is nonnegative. 
\end{remark}
\begin{proof}
By the equivalence \eqref{E3.105}, 
it suffices to prove in the case $d=2$. 
Once \eqref{E3.12} is established, \eqref{Ea3.13} 
follows immediately from Lemma \ref{L1}, 
since $\la_n(\xi_1,\xi_2) > \la_0 (\frac 12 , \xi_2) \ge \frac 14 \te_0$. 
Thus we focus on the proof of \eqref{E3.12}. 
For this sake, we may suppose 
$0\le \xi_1\le \frac 12$ by \eqref{E3.11}, 
and hence $|\xi_1|_* = \xi_1$. 
We divide the proof into two cases. \par
\textit{Case I}. Suppose that $0\le \xi_1\le \frac 1{16}$. 
Since $A_\xi$ is self-adjoint on $L^2(\T)$, 
it suffices to show that for some $\te_0>0$, we have
\begin{equation}\label{Ec3.13}
\inp{A_\xi u}{u}_{L^2(\T)} \ge \te_0 |\xi|^2 \nor{u}{L^2(\T)}^2
\end{equation}
for all $u\in H^2(\T)$ (see Section V.10 in \cite{MR1335452}). 
First, we have
\begin{equation}\label{Eb3.13}
    \begin{aligned}
      \inp{A_\xi u}{u}_{L^2(\T)}
       & =
      \inp{
        (-(\rd_{x_1} + i \xi_1)^2 + \xi_2^2 - \xi_2 \cos x_1) u
      }{u}_{L^2(\T)} \\
       & =
      \nor{(\rd_{x_1} + i\xi_1) u}{L^2(\T)}^2 + \xi_2^2 \nor{u}{L^2(\T)}^2
      -\xi_2 
      \int_{0}^{2\pi} \cos x_1 |u|^2 dx_1.
    \end{aligned}
\end{equation}
Now we decompose
\[
    u = a_0 \psi_0 + v,\qquad \text{with}\quad \psi_0 = \frac{1}{\sqrt{2\pi}},\quad \psi_0 \perp v,
\]
where $f\perp g$ for $f,g\in L^2(\T)$ means $\inp{f}{g}_{L^2(\T)} =0$. 
We further write
\begin{equation*}
    v = a_1 \psi_1 + a_{-1} \psi_{-1} + w,\qquad
    \text{with}\quad
    \psi_{\pm 1} = \frac{1}{\sqrt{2\pi}} e^{\pm ix_1}, \quad
    w \perp \psi_0, \psi_1, \psi_{-1}.
\end{equation*}
By modifying the phase of $u$, we may assume $a_0\in \R$ without loss of generality. Then we have
\[
(\rd_{x_1} +i\xi_1)u = i\xi_1a_0 \psi_0 
+ i (1+\xi_1)a_1 \psi_1 + i(-1+\xi_1)a_{-1} \psi_{-1} + (\rd_{x_1} + i\xi_1)w.
\]
Noting that $(\rd_{x_1} + i\xi_1)w \perp \psi_0,\psi_1,\psi_{-1}$, we obtain
\begin{align*}
&\nor{(\rd_{x_1} + i\xi_1) u}{L^2(\T)}^2\\
&=
\xi_1^2 a_0^2 + (1+\xi_1)^2 |a_1|^2 + (-1+\xi_1)^2 |a_{-1}|^2 +
\nor{(\rd_{x_1} + i\xi_1) w}{L^2(\T)}^2.
\end{align*}
Moreover, we have
\begin{gather*}
\nor{u}{L^2(\T)}^2 = a_0^2 + |a_1|^2 + |a_{-1}|^2 + \nor{w}{L^2(\T)}^2,\\
\begin{aligned}
\int_{0}^{2\pi} \cos x_1 |u|^2 dx_1
&=
\int_{0}^{2\pi} \cos x_1 |a_0 \psi_0 + v|^2 dx_1 \\ 
&=
\int_{0}^{2\pi} \cos x_1 \left(\sqrt{\frac 2\pi} a_0 \Re v + |v|^2 \right) dx_1.
\end{aligned}
\end{gather*}
Now we observe that \eqref{Eb3.13} implies
\begin{align*}
      \cos x_1 \Re v
       & =
      \frac{e^{ix_1} + e^{-ix_1}}{2} 
      \Re (a_1 \psi_1 + a_{-1} \psi_{-1}) 
      + \cos x_1 \Re w
      \\
       & =
      \frac{1}{\sqrt{2\pi}}
      \frac{a_1 + \ovl{a_{-1}}}{4}
      (e^{2ix_1} +1)
      +
      \frac{1}{\sqrt{2\pi}}
      \frac{a_{-1} + \ovl{a_{1}}}{4}
      (e^{-2ix_1} + 1)
      +
      \cos x_1 \Re w.
\end{align*}
Since $w\perp \psi_{\pm 1}$ in $L^2(\T)$, 
we have $\int_0^{2\pi} \cos x_1 \Re w\, dx_1=0$. By noting this, it follows that 
\[
\int_0^{2\pi} \cos x_1 \Re v dx_1
=
\sqrt{2\pi}
\frac{a_1 + \ovl{a_{-1}}}{4}
+
\sqrt{2\pi} \frac{a_{-1} + \ovl{a_{1}}}{4}
=
\sqrt{\frac{\pi}{2}} \left( \Re a_1 + \Re a_{-1} \right).
\]
Consequently, \eqref{Eb3.13} is rewritten as
\begin{align*}
\inp{A_\xi u}{u}_{L^2(\T)}
=&
\xi_1^2 a_0^2 + (1+\xi_1)^2 |a_1|^2 + (-1 + \xi_1)^2 |a_{-1}|^2
+ \nor{(\rd_{x} + i\xi_1)w}{L^2(\T)}^2 \\
& + \xi_2^2 \left( a_0^2 + |a_1|^2 + |a_{-1}|^2 + \nor{w}{L^2(\T)}^2 \right) \\
&-\xi_2 a_0\left(\Re a_1 + \Re a_{-1}\right)
- \xi_2 \int_0^{2\pi}\cos x_1 |v|^2 dx_1.
\end{align*}
Here, let $\eps\in (0,1)$ be a number determined later. Then, Young's inequality yields
\[
\left|\xi_2 a_0 
\Re a_{\pm 1}  \right|
\le 
\frac 12 \left(
(1-\eps)  \xi_2^2 a_0^2 + \frac{1}{1-\eps} |a_{\pm 1}|^2
\right).
\]
Moreover, since $w\perp \psi_0,\psi_{\pm 1}$ in $L^2(\T)$, the theory of Fourier series gives
\[
\nor{(\rd_{x_1} + i\xi_1) w}{L^2(\T)}^2 \ge (2-\xi_1)^2 \nor{w}{L^2(\T)}^2.
\]
Thus
\begin{align*}
\inp{A_{\xi} u}{u}_{L^2(\T)}
     & \ge
    \left((\xi_1+1)^2 - \frac{1}{2(1-\eps)}\right) |a_1|^2
    + \left((\xi_1-1)^2 - \frac{1}{2(1-\eps)}\right) |a_{-1}|^2\\
    & \hspace{30pt}+(2-\xi_1)^2 \nor{ w}{L^2(\T)}^2  + \xi_1^2 a_0^2
    \\
     & \hspace{30pt}
    + \xi_2^2 \left( \eps a_0^2 +  |a_1|^2 +  |a_{-1}|^2 + \nor{w}{L^2(\T)}^2\right)
    - \xi_2 \int_0^{2\pi}
    \cos x_1 |v|^2 dx_1.
\end{align*}
Since $\nor{v}{L^2(\T)}^2 = |a_1|^2+|a_{-1}|^2 + \nor{w}{L^2(\T)}^2$ and $0\le \xi_1 \le \frac 1{16}$, the first line is bounded from below by
\[
\left((\xi_1-1)^2 - \frac{1}{2(1-\eps)}\right) \nor{v}{L^2(\T)}^2.
\]
Hence we obtain
\begin{align*}
\inp{A_{\xi_1, \xi_2} u}{u}_{L^2(\T)}
     &\ge 
    \xi_1^2 a_0^2
    + \eps \xi_2^2 \left( a_0^2 +  |a_1|^2 +  |a_{-1}|^2 + \nor{w}{L^2(\T)}^2\right) \\
     & \qquad  +\left(
(\xi_1-1)^2 - \frac{1}{2(1-\eps)} + (1-\eps) \xi_2^2 -\xi_2
\right) \nor{v}{L^2(\T)}^2. 
\end{align*}
Now, let us set $\eps= \frac 18$. Then by noting $0\le \xi_1 \le \frac 1{16}$, we have
\begin{align*}
(1-\eps) \xi_2^2 - \xi_2 +
(\xi_1-1)^2 -\frac{1}{2(1-\eps)}
&= \frac 78 \left(\xi_2 - \frac 47\right)^2
+\xi_1^2 + \left( - 2\xi_1 + \frac 17 \right)
\\
& \ge \xi_1^2.
\end{align*}
Therefore, we obtain
\begin{align*}
\inp{A_{\xi_1, \xi_2} u}{u}_{L^2(\T)} 
&\ge
    \xi_1^2 a_0^2
    + \eps \xi_2^2 \left( a_0^2 +  |a_1|^2 +  |a_{-1}|^2 + \nor{w}{L^2(\T)}^2\right)
+ \xi_1^2\nor{v}{L^2(\T)}^2\\
&\ge \frac 18 |\xi|^2 \nor{u}{L^2(\T)}^2,
\end{align*}
which proves \eqref{Ec3.13}.\par
\textit{Case II.} 
Suppose that $\frac 1{16} \le \xi_1 \le \frac 12$. By Lemma \ref{L1}, combined with 
the result of Case I, we have
\[
\la_0(\xi_1,\xi_2) \ge \la_0(\frac 1{16}, \xi_2) \ge 
\frac 18 \left( \frac 1{16^2} + \xi_2^2\right)
\ge 
\frac 18 \left( \frac {\xi_1^2}{8^2} + \xi_2^2\right)
\ge \te_0 |\xi|^2
\]
for some $\te_0>0$. Hence the proof is complete.
\end{proof}
\subsection{Asymptotics near $\xi=0$}
Next, we examine the asymptotics of $\la_0(\xi)$ and $\phi_0(\xi,\cdot)$ around $\xi=0$.
Since $A_0 = -\rd_{x_1}^2$, 
we know that $\la_0(0)=0$ with $\phi_0 (0,\cdot) = \frac 1{\sqrt{2\pi}}$, which leads us to apply the Lyapunov--Schmidt argument. We claim the following.
\begin{prop}\label{P3}
There exist $\del_0>0$, $C_k>0$ for $k\in\Z_{\ge0}$ such that for any $\xi\in\T^*$ with $|\xi|_*<\del_0$, we have
\begin{gather}
\label{E3.13}
      \la_0(\xi) = |\xi_1|_*^2 + \frac 12 |\xi'|^2 + o(|\xi|_*^2),
\\
\label{E3.14}
\phi_0(\xi,x_1) = 
\begin{cases}
\frac{1}{\sqrt{2\pi}} + \frac{1}{\sqrt{2\pi}} \xi_2 \cos x_1 + R_0(\xi,x_1), & \text{if } d=2,\\
\frac{1}{\sqrt{2\pi}} + \frac{1}{\sqrt{2\pi}} \xi_2 \cos x_1 +
\frac{1}{\sqrt{2\pi}} \xi_3 \sin x_1
+ R_0(\xi,x_1), & \text{if } d=3,
\end{cases}
\end{gather}
with 
\begin{equation}\label{E3.15}
|\rd^k_{x_1} R_0 (\xi,x_1)| \le C_k |\xi|_*^2,\quad 
x_1\in\T,\quad k\ge 0.
\end{equation}
\end{prop}
\begin{proof}
By \eqref{E3.11}, we may restrict ourselves to $\xi_1\in [0,\frac 12]$, where $|\xi_1|_* = \xi_1$ holds. We divide the proof into three steps. \par
\textit{Step 1.} 
We start with the case $d=2$. 
For $A_\xi$, we seek an eigenvalue $\tilde\la_0(\xi)$ and the associated eigenfunction $\tilde{\phi}_0(\xi)$ 
in the following form:
\begin{gather}\label{E3.16}
  A_\xi \tilde{\phi}_0 (\xi,\cdot) = \tilde{\la}_0 (\xi) \tilde{\phi}_0 (\xi,\cdot),\qquad 
\tilde\la_0(\xi) = O(|\xi|^2), \\
\label{E3.17}
\tilde{\phi}_0 (\xi,\cdot) = \psi_0 + V_\xi ,\qquad
\psi_0 = \frac 1{\sqrt{2\pi}},\qquad 
V_\xi \perp \psi_0 \text{ in } L^2(\T).
\end{gather}
We first derive an equivalent form of \eqref{E3.16} and \eqref{E3.17}.  
To begin with, 
substituting \eqref{E3.17} into \eqref{E3.16}, we have
\begin{equation}\label{E3.18}
 (|\xi|^2 - \xi_2 \cos x_1) \psi_0 
+ \left(-\rd_{x_1}^2 -2i\xi_1\rd_{x_1} + |\xi|^2 - \xi_2 \cos x_1 \right) V_\xi
= \tilde\la_0(\xi) (\psi_0 + V_\xi).
\end{equation}
Let $P_0= \inp{\cdot}{\psi_0}_{L^2(\T)} \psi_0$ be the projection onto $\psi_0$-component. 
Applying $P_0$ to \eqref{E3.18}, we have
\[
|\xi|^2 \psi_0 -\xi_2 P_0 \left( V_\xi \cos x_1  \right) = \tilde\la_0 (\xi) \psi_0,
\]
which gives
\begin{equation}\label{E3.19}
\tilde\la_0(\xi) =  |\xi|^2 - \frac{\xi_2}{\sqrt{2\pi}} \int_0^{2\pi} V_\xi(x_1) \cos x_1 dx_1.
\end{equation}
Next, we apply $P_0^\perp = I -P_0$ to \eqref{E3.18}. Then we obtain
\[
-\rd_{x_1}^2 V_\xi - \xi_2 \cos x_1 \psi_0 
+ (- 2i\xi_1 \rd_{x_1}  + |\xi|^2) V_\xi
    - \xi_2 P_0^\perp (V_\xi \cos x_1) = \tilde\la_0 (\xi) V_\xi .
\]
Noting that $-\rd_{x_1}^2$ is invertible on $P_0^\perp L^2(\T)$, 
we can rewrite this into
\begin{equation}\label{E3.20}
\left( 1 +  (-\rd_{x_1}^2)^{-1} P_0^\perp (-\tilde\la_0(\xi) + |\xi| B_\xi)\right) V_\xi =\frac 1{\sqrt{2\pi}} \xi_2 \cos x_1,
\end{equation}
where
\[
B_\xi = -2i \frac{\xi_1}{|\xi|} \rd_{x_1} + |\xi| - \frac{\xi_2}{|\xi|} \cos x_1.
\]
Here, we note that 
\[
\nor{(-\rd_{x_1}^2)^{-1} P_0^\perp}{L^2(\T)\to L^2(\T)}\le 1,\qquad 
\nor{(-\rd_{x_1}^2)^{-1} P_0^\perp \rd_{x_1}}{L^2(\T)\to L^2(\T)}\le 1.
\]
Hence, if $|\xi|$ and $\tilde\la_0(\xi)$ are sufficiently small, which will be justified a posteriori, we have
\begin{equation}\label{Ea3.20}
\nor{(-\rd_{x_1}^2)^{-1} P_0^\perp (-\tilde\la_0(\xi) + |\xi| B_\xi)}{L^2(\T)\to L^2(\T)} \le \frac 12.
\end{equation}
Then the operator in the right hand side of \eqref{E3.20} is invertible, yielding
\begin{equation}\label{E3.21}
  V_\xi =
  \frac{\xi_2}{\sqrt{2\pi}}
  \sum_{j=0}^\infty (-1)^j
  \left[(-\rd_{x_1}^2)^{-1} P_0^\perp (-\tilde\la_0(\xi) + |\xi| B_\xi)\right]^j \cos x_1.
\end{equation}
Writing \eqref{E3.19} and \eqref{E3.21} as
\begin{equation}\label{E3.22}
  \tilde\la_0(\xi) = S_\la [\tilde\la_0(\xi), V_\xi],\qquad V_\xi = S_V [\tilde\la_0(\xi), V_\xi],
\end{equation}
we claim that $S=(S_\la,S_V)$ is a contraction mapping on
\[
\boM =
\left\{
(\la, V) \in \R \times L^2(\T)\ \left|\
|\la| \le 2|\xi|^2, \ \nor{V}{L^2(\T)} \le \sqrt{2}|\xi|
\right.
\right\}.
\]
for $|\xi|\le \del_0$ with sufficiently small $\del_0>0$. First, let $(\la,V)\in\boM$. 
Then we have
\[
  S_\la [\la ,V]
  \le |\xi|^2 + \frac{|\xi_2|}{\sqrt{2\pi}} \nor{V}{L^2(\T)} \nor{\cos x_1}{L^2(\T)}
  \le |\xi|^2 +\frac 1{\sqrt{2}} |\xi_2| \nor{V}{L^2(\T)} \le 2|\xi|^2.
\]
On the other hand, choosing $\del_0>0$ sufficiently small so that 
\eqref{Ea3.20} holds true, we estimate
\begin{align*}
  S_V [\la, V]
  &\le \frac{|\xi_2|}{\sqrt{2\pi}}
  \sum_{j=0}^\infty \nor{(-\rd_{x_1}^2)^{-1} P_0^\perp (\la+|\xi|B_\xi) }{L^2(\T)}^j
  \nor{\cos x}{L^2(\T)}\\
  &\le \frac{|\xi_2|}{\sqrt{2}} \sum_{j=0}^\infty 2^{-j} =\sqrt{2} |\xi|.
\end{align*}
Hence $S$ maps $\boM$ into itself. 
Next, let $(\la, V)$, $(\tilde{\la}, \tilde{V})\in\boM$. 
Using \eqref{E3.20}, we can write
\begin{align*}
&S_V[\la,V] - S_V [\tilde\la,\tilde{V}] \\
&=
\left(1+ (-\rd_{x_1}^2)^{-1} P_0^\perp (-\la + |\xi| B_\xi)\right)^{-1} \left[ (\tilde\la-\la)
(-\rd_{x_1}^2)^{-1} P_0^\perp 
S_V[\tilde\la,\tilde{V}]
\right].
\end{align*}
Therefore,
\begin{align*}
&\nor{S_V[\la,V] - S_V [\tilde\la,\tilde{V}]}{L^2(\T)} \\
&\le |\la - \tilde{\la}|
\sum_{j=0}^\infty 2^{-j}
\nor{(-\rd_{x_1}^2)^{-1} P_0^\perp}{L^2(\T)\to L^2(\T)} \nor{S_V[\tilde\la,\tilde{V}]}{L^2(\T)} 
\le 2\sqrt{2} |\xi|  |\la - \tilde{\la}|.
\end{align*}
We also have
\begin{align*}
  \left|
  S_\la [\la ,V] - S_\la [\tilde{\la} , \tilde{V}]
  \right|
  &=
  \left| \frac{\xi_2}{\sqrt{2\pi}} \int_0^{2\pi} (V(x_1) - \tilde{V} (x_1)) \cos x_1 dx_1 \right| \\
  &\le \frac{|\xi|}{\sqrt{2}} \nor{V- \tilde{V}}{L^2(\T)}.
\end{align*}
Hence, $S$ is a contraction on $\boM$ if $\del_0$ is sufficiently small. 
Consequently, \eqref{E3.22} has a unique solution $(\la_0(\xi), V_\xi) \in \boM$.
\par
\textit{Step 2.} 
We shall conclude the statement when $d=2$. 
Suppose $|\xi|\le \del_0$, and let $(\tilde\la_0(\xi), V_\xi) \in \boM$ be the solution to \eqref{E3.22} as above. 
We first note that $\tilde\la_0(\xi)$ should coincide with $\la_0(\xi)$ if $\del_0$ is chosen further small, since $\tilde\la_0(\xi)\le 2\del_0^2$ and any other eigenvalues are larger than $2\del_0^2$ by virtue of \eqref{Ea3.13}. 
Since $\la_0(\xi) \le 2|\xi|^2$, we have
\begin{equation}\label{Ea3.22}
  \nor{(-\rd_{x_1}^2)^{-1} P_0^\perp (-\la_0(\xi) + |\xi| B_\xi)}{H^k (\T)\to H^k(\T)}
  \le C_k(|\xi| + \la_0(\xi)) \le C_k |\xi|
\end{equation}
for $k\ge 0$ with some $C_k>0$ independent of $\xi$. 
Here, \eqref{E3.21} can be rewritten as
\begin{equation}\label{E3.23}
V_\xi = \frac{\xi_2}{\sqrt{2\pi}} \cos x_1 + \tilde{R}_0(\xi,x_1)
\end{equation}
with
\begin{equation}\label{E3.24}
\tilde{R}_0(\xi,x_1) =  \frac{\xi_2}{\sqrt{2\pi}} \sum_{j=1}^\infty (-1)^j \left[
    (-\rd_{x_1}^2)^{-1} P_0^\perp (-\la_0(\xi) + |\xi| B_\xi)
    \right]^j \cos x_1. 
\end{equation}
Then from \eqref{Ea3.22}, it follows that 
\begin{equation}\label{E3.29}
  \nor{ \tilde{R}_0(\xi,\cdot)}{H^k(\T)}
  \le C_k |\xi|^2,\qquad k\ge 0.
\end{equation}
if $\del_0$ is sufficiently small. Therefore, inserting \eqref{E3.23} to \eqref{E3.19}, we obtain
\[
  \la_0(\xi) = |\xi|^2 - \frac{\xi_2}{\sqrt{2\pi}} \int_0^{2\pi}
  \frac{\xi_2}{\sqrt{2\pi}} \cos^2 x_1 dx_1 + O(|\xi|^3)
  = \xi_1^2 + \frac 12 \xi_2^2 + O(|\xi|^3),
\]
as claimed in \eqref{E3.13}. 
On the other hand, since $\phi_0(\xi,\cdot)$ is the $L^2$-normalized eigenfunction, we have 
$\phi_0 (\xi,\cdot) = \tilde{\phi}_0(\xi,\cdot) / \nor{\tilde{\phi}_0(\xi,\cdot)}{L^2(\T)}$. 
By noting that
\[
  \nor{\tilde{\phi}_0 (\xi,\cdot)}{L^2(\T)}^2 = \nor{\psi_0 + V_\xi}{L^2(\T)}^2
  = 1 + \nor{V_\xi}{L^2(\T)}^2 = 1+O(|\xi|^2),
\]
we obtain
\[
  \phi_0 (\xi, x_1) = (1+O(|\xi|^2)) \tilde{\phi}_0 (\xi, x_1)
  = \frac{1}{\sqrt{2\pi}} + \frac{1}{\sqrt{2\pi}} \xi_2 \cos x_1 + R_0 (\xi,x_1),
\]
where $R_0(\xi,\cdot)$ satisfies the same bound as \eqref{E3.29}. 
This is the desired formula as claimed in \eqref{E3.14}, while 
\eqref{E3.15} follows from the Sobolev inequality. \par
\textit{Step 3.} We shall conclude the statement when $d=3$. 
By \eqref{E3.105}, we obtain \eqref{E3.13} from the conclusion in the case $d=2$. 
Moreover, \eqref{E3.105} and 
the results of Step 2 imply
\begin{align*}
\phi_0(\xi,x_1) 
&= \frac{1}{\sqrt{2\pi}} + 
\frac{1}{\sqrt{2\pi}} |\xi'| 
\boT_{\al_{\xi'}} \cos x_1 
+ \boT_{\al_{\xi'}} R_0(\xi,x_1)\\
&= 
\frac{1}{\sqrt{2\pi}} + 
\frac{1}{\sqrt{2\pi}} \xi_2\cos x_1 
+\frac{1}{\sqrt{2\pi}} \xi_3\sin x_1 +  \boT_{\al_{\xi'}} R_0(\xi,x_1).
\end{align*}
Thus \eqref{E3.14} in the case $d=3$ follows, by rewriting the last term as $R_0(\xi,x_1)$.
\end{proof}
%
%
%
%
%
%
%
%
\section{Linear estimates}\label{S4}
In this section, we consider the corresponding linear equation of \eqref{E2.4}:
\begin{equation}\label{E3.1}
\left\{
\begin{aligned}
\rd_t v  &= (-\al + i) A v + F,\qquad & (t,x)\in (0,\infty) \times \R^d,\\
v (0,x) &= v_0 (x) , & x\in \R^d,
\end{aligned}
\right.
\end{equation}
with $v=v(t,x), F= F(t,x)\colon [0,\infty)\times \R^d \to\C$ and $v_0 = v_0(x)\colon \R^d\to \C$. 
Equivalently, \eqref{E3.1} can be rewritten as 
\begin{equation}\label{E4.2}
v(t,x) = e^{t(-\al + i)A} v_0  + \int_0^t e^{(t-s)(-\al+i)A} F(s,x) ds,
\end{equation}
where $e^{t(-\al + i)A}$ is the semigroup generated by $(-\al+i)A$. Note also that 
using $\BF$, we can represent $e^{t(-\al + i)A}$ as 
\begin{equation}\label{E4.3}
e^{t(-\al + i)A} v_0 =
\int_{\T^*} e^{i\xi\cdot x} e^{t(-\al +i)A_\xi} \BF [v_0] (\xi,x_1) d\xi.
\end{equation}
The goal of this section is to obtain time decay estimates for $v$. 
We restrict our attention to the nontrivial cases $d=2,3$; 
the case $d=1$ is standard as $A$ coincides with the Laplacian, and the required estimates are summarized later. 
Following the idea of \cites{MR2601067,MR2788923}, 
we divide \eqref{E4.2} and \eqref{E4.3} into 
high- and low-frequency parts in terms of $\T^*$, 
and obtain the estimates separately. 
For this sake, we start with introducing the operators which restrict a function onto each frequency part. 
Let $\chi\in C^\infty (\T^*)$ be a function with
\[
0\le \chi(\xi) \le 1 \text{ in } \T^*,\qquad 
\chi (\xi) =1 \text{ if } |\xi|_*<\frac{\del_0}2,\qquad 
\chi (\xi) =0 \text{ if } |\xi|_*>\del_0,
\]
where $\del_0$ is the number as in Proposition \ref{P3}. 
Let $P_n(\xi)$ be the eigenprojection of $A_\xi$ associated with the eigenvalue $\la_n(\xi)$; namely,
\[
P_n(\xi) f = \inp{f}{\phi_n(\xi,\cdot)}_{L^2(\T)} \phi_n(\xi,\cdot),\qquad f\in L^2(\T).
\]
Then, 
for $v\in L^2(\R^d)$, define
\begin{gather}\label{ea6.1}
Q_L v = 
\int_{\T^*} e^{i\xi\cdot x} \chi (\xi) P_0(\xi) \BF [v] (\xi,x_1) d\xi,
\\
\label{ea6.2}
Q_H v = (I -Q_L)v
= \int_{\T^*} e^{i\xi\cdot x} (1-\chi (\xi) P_0(\xi)) \BF [v] (\xi,x_1) d\xi.
\end{gather}
It is clear that by \eqref{E3.3}, $Q_L$, $Q_H$ are bounded operators on $L^2(\R^d)$. 
Since both $Q_L$ and $Q_H$ commute with $A$, 
the linear evolution is completely decoupled into $Q_L$ and $Q_H$ parts. 
\subsection{Preliminary lemmas}
We begin by preparing some lemmas concerning the quantitative properties of $A$.
\begin{lem}\label{L4}
For any integer $s\ge 0$, 
there exists $C = C (s)>0$ such that
\begin{equation}\label{7.1}
C^{-1} \nor{v}{H^s(\R^d)} \le 
\nor{(1+A)^{\frac s2} v}{L^2(\R^d)} \le C \nor{v}{H^s(\R^d)}
\end{equation}
for $v\in H^s(\R^d)$.
\end{lem}
\begin{proof}
We provide the proof for $d=2$, while the case $d=3$ can be proved similarly. 
We proceed by induction on $s$.  
When $s=0$, \eqref{7.1} is trivial. 
When $s=1$, we first obtain
\begin{align*}
\nor{(1+A)^{\frac 12}v}{L^2}^2 
&= \inp{(1-\Del + i\cos x_1 \rd_{2} ) v}{v}_{L^2} \\
&= \nor{v}{H^1}^2 + \inp{i\cos x_1 \rd_{2} v}{v}_{L^2} 
\le C \nor{v}{H^1}^2.
\end{align*}
For the other inequality, we note that
\begin{align*}
\nor{v}{H^1}^2 
= \inp{(1-\Del) v}{v}_{L^2} 
&=\inp{(1+A)v}{v}_{L^2} - \inp{i\cos x_1\rd_{2} v}{v}_{L^2} \\
&\le  \nor{(1+A)^{\frac 12}v}{L^2}^2 + \nor{\rd_2 v}{L^2} \nor{v}{L^2}\\
&\le \nor{(1+A)^{\frac 12}v}{L^2}^2 + \frac 12 \nor{v}{H^1}^2 + C \nor{v}{L^2}^2.
\end{align*}
Hence we have
\[
\nor{v}{H^1}^2 \le C \nor{(1+A)^{\frac 12}v}{L^2}^2 + C \nor{v}{L^2}^2 \le C \nor{(1+A)^{\frac 12}v}{L^2}^2,
\]
where the last inequality follows by
\[
\nor{(1+A)^{1/2}v}{L^2}^2 
= \nor{v}{L^2}^2 + \inp{Av}{v}_{L^2}
\ge \nor{v}{L^2}^2.
\]
Thus, \eqref{7.1} holds true when $s=1$. 
We now suppose that \eqref{7.1} is true for $s\le 2k-1$ with $k\in \N$. 
Then, for $a=0,1$, it follows from the assumption that
\[
\nor{(1+A)^{k+ \frac a2} v}{L^2} \le C 
\nor{(1+A)v}{H^{2k+a-2}}
\le C
\nor{v}{H^{2k+a}}.
\]
On the other hand, the assumption, together with the Leibniz rule and the interpolation inequality, yields
\begin{align*}
\nor{v}{H^{2k+a}} \le C\nor{(1-\Del) v}{H^{2k+a-2}} 
&\le C \nor{(1+A) v}{H^{2k+a-2}} + C\nor{i\cos x_1 \rd_2 v}{H^{2k+a-2}}\\
&\le C \nor{(1+A)^{k+\frac a2} v}{L^2} + C \nor{v}{H^{2k+a-1}}\\
&\le C \nor{(1+A)^{k+\frac a2} v}{L^2} + \frac 12 \nor{v}{H^{2k+a}} + C \nor{v}{L^2},
\end{align*}
which is equivalent to
\begin{align*}
\nor{v}{H^{2k+a}} 
&\le C\nor{(1-\Del) v}{H^{2k+a-2}} \\
&\le C \nor{(1+A)^{k+\frac a2} v}{L^2} + C \nor{v}{L^2} \le C \nor{(1+A)^{k+\frac a2} v}{L^2}.
\end{align*}
Therefore, \eqref{7.1} holds true for $s=2k, 2k+1$. Hence the proof is complete.
\end{proof}
\begin{lem}\label{L3}
The following hold.\\
(i) For $s\ge 0$, there exists $C=C(s)>0$ such that for all $v\in H^s(\R^d)$, we have
\begin{equation}\label{Ea4.4}
\nor{Q_L v}{H^s(\R^d)} \le C \nor{Q_L v}{L^2(\R^d)}.
\end{equation}
(ii) There exists $C_1>0$ such that for all $v\in H^1(\R^d)$
, we have
\begin{equation}\label{Ea4.5}
\nor{Q_H v}{L^2(\R^d)} \le C_1 \nor{A^{\frac 12} Q_H v}{L^2(\R^d)}.
\end{equation}
\end{lem}
\begin{proof}
(i) immediately follows from \eqref{7.1} and Proposition \ref{P2}, since
\begin{align*}
\nor{(1 + A)^{\frac s2} Q_L v}{L^2(\R^d)} 
&= 
\frac{1}{(2\pi)^{d-1}} \nor{(1+\la_0(\xi))^{\frac s2} \chi(\xi)
P_0(\xi) \BF [v] (\xi,x_1)}{L^2_{\xi}L^2_{x_1}}
\\
&\le C(s) \nor{\chi(\xi)
P_0(\xi) \BF [v] (\xi,x_1)}{L^2_\xi L^2_{x_1}}\\
&\le C(s) \nor{Q_L v}{L^2(\R^d)}.
\end{align*}
Next, applying the eigenfunction expansion yields
\begin{align*}
A^{\frac 12 } Q_H v 
&= 
\int_{\T^*} e^{i\xi\cdot x} 
(1-\chi(\xi) P_0(\xi)) A_\xi^{\frac 12} \BF[v](\xi,x_1)d\xi 
\\
&=
\int_{\T^*} e^{i\xi\cdot x} (1-\chi (\xi) P_0 (\xi)) \sum_{n=0}^\infty \la_n (\xi)^{\frac 12} P_n(\xi) \BF [v] (\xi,x_1) d\xi.
\end{align*}
Hence, we have
\begin{align*}
&\nor{A^{\frac 12} Q_H v}{L^2(\R^d)}^2 \\
&=
\frac{1}{(2\pi)^{d-1}} \nor{A_\xi^{\frac 12 } (1-\chi (\xi) P_0 (\xi)) \sum_{n=0}^\infty 
P_n(\xi) \BF [v] (\xi,x_1)}{L^2_{\xi}L^2_{x_1}}^2 \\
&= 
\frac{1}{(2\pi)^{d-1}} \nor{(1-\chi (\xi) ) \la_0 (\xi)^{\frac 12} P_0(\xi) \BF [v] (\xi,x_1)}{L^2_{\xi}L^2_{x_1}}^2 \\
&\qquad + \frac{1}{(2\pi)^{d-1}}
\sum_{n=1}^\infty
\nor{ \la_n (\xi)^{\frac 12} P_n(\xi) \BF [v] (\xi,x_1)}{L^2_{\xi}L^2_{x_1}}^2 \\
&\ge 
C_1^{-2} \left[
\nor{(1-\chi (\xi) ) P_0(\xi) \BF [v] (\xi,x_1)}{L^2_{\xi}L^2_{x_1}}^2 
+
\sum_{n=1}^\infty
\nor{ P_n(\xi) \BF [v] (\xi,x_1)}{L^2_{\xi}L^2_{x_1}}^2 \right] \\
&= C_1^{-2} \nor{Q_H v}{L^2_{\xi}L^2_{x_1}}^2,
\end{align*}
where $C_1$ is a constant satisfying
\[
\la_0 (\xi) \ge \frac{(2\pi)^{d-1}}{C_1^2}  \text{ if } |\xi|_*\ge \frac{\del_0}{2}, 
\quad \text{and}\quad 
\la_n (\xi) \ge \frac{(2\pi)^{d-1}}{C_1^2} \text{ for } n\ge 1,\ \xi\in \T^*,
\]
whose existence is guaranteed by Proposition \ref{P2}.
\end{proof}
\subsection{High-frequency estimate}
For the high-frequency part, we claim the following:
\begin{prop}[High-frequency estimate]\label{P:4}
Let $v$ be a solution to \eqref{E3.1}. Then, for any integer $s\ge 1$, 
there exist $c_0>0$ and $C=C(s)>0$, independent of $\al$ and $v$, such that,
\begin{equation}\label{e7.2}
\begin{aligned}
&\nor{Q_H v(t)}{H^s(\R^d)} + \al^{\frac 12} 
\left[\int_0^t e^{-c_0 \al(t-\tau)} \nor{Q_H v(\tau)}{H^{s+1}(\R^d)}^2 d\tau\right]^{\frac 12}\\
&\qquad \le 
C e^{-\frac {c_0}2 \al t} \nor{Q_H v_0}{H^s(\R^d)} 
+ C \al^{-\frac 12}
\left[\int_0^t e^{-c_0 \al(t-\tau)} \nor{Q_H F(\tau)}{H^{s-1}(\R^d)}^2 d\tau\right]^{\frac 12}.
\end{aligned}
\end{equation}
\end{prop}
\begin{proof}
Since $Q_H v $ solves
\[
\rd_t Q_H v = (-\al +i) A Q_H v + Q_H F,
\]
we have
\begin{align*}
\frac 12 \frac d{dt} \nor{(1+A)^{\frac s2} Q_H v(t)}{L^2}^2
&=
\Re \inp{(1+A)^{\frac s2 } \rd_t Q_H v}{(1+A)^{\frac s2 } Q_H v }_{L^2} \\
&=
-\al \nor{A^{\frac 12} (1+A)^{\frac s2} Q_H v}{L^2}^2 \\
&\qquad + \Re \inp{(1+A)^{\frac {s-1}2 } Q_H F }{(1+A)^{\frac {s+1}2 } Q_H v}_{L^2}.
\end{align*}
Here, let $C_1$ be the constant in Lemma \ref{L3} (ii). Then for $f\in L^2$, we have
\[
\nor{(1+A)^{\frac 12} Q_H f}{L^2}^2 
= \nor{ Q_H f}{L^2}^2 + \nor{A^{\frac 12} Q_H f}{L^2}^2 \le (C_1^2 +1) \nor{A^{\frac 12} Q_H f}{L^2}^2.
\]
Therefore, letting $c_0 = (C_1^2 +1)^{-1}$, we have
\begin{align*}
&\frac 12 \frac d{dt} \left( e^{c_0 \al t} \nor{(1+A)^{\frac s2} Q_H v}{L^2}^2 \right)
+ \frac {c_0 \al}{2} e^{c_0 \al t} \nor{(1+A)^{\frac {s+1}2} Q_H v}{L^2}^2 \\
&\le c_0 \al e^{c_0 \al t} \nor{(1+A)^{\frac {s+1}2} Q_H v}{L^2}^2 +
\frac 12 e^{c_0 \al t} \frac d{dt} \nor{(1+A)^{\frac s2} Q_H v}{L^2}^2 \\
&\le \al e^{c_0 \al t} \nor{A^{\frac 12} (1+A)^{\frac s2} Q_H v}{L^2}^2  + 
\frac 12 e^{c_0 \al t} \frac d{dt} \nor{(1+A)^{\frac s2} Q_H v}{L^2}^2 \\
&= e^{c_0 \al t} \Re \inp{(1+A)^{\frac {s-1}2} Q_H F }{(1+A)^{\frac {s+1}2} Q_H v}_{L^2} \\
&\le \frac{1}{c_0 \al} e^{c_0 \al t} \nor{(1+A)^{\frac {s-1}2} Q_H F}{L^2}^2 
+ \frac{c_0 \al}{4} e^{c_0\al t} \nor{(1+A)^{\frac {s+1}2} Q_H v}{L^2}^2.
\end{align*}
Hence, it follows that
\begin{align*}
&\frac 12 \frac d{dt} \left( e^{c_0 \al t} \nor{(1+A)^{\frac s2} Q_H v}{L^2}^2 \right)
+ \frac {c_0 \al}{4} e^{c_0 \al t} \nor{(1+A)^{\frac {s+1}2} Q_H v}{L^2}^2 \\
&\le \frac{1}{c_0 \al} e^{c_0 \al t} \nor{(1+A)^{\frac {s-1}2} Q_H F}{L^2}^2 .
\end{align*}
Integrating over $t$, we obtain
\begin{align*}
&\frac 12 e^{c_0 \al t} \nor{(1+A)^{\frac s2} Q_H v(t)}{L^2}^2 
+ \frac{c_0\al}{4} \int_0^t e^{c_0 \al \tau} \nor{(1+A)^{\frac{s+1}2} Q_H v(\tau)}{L^2}^2 d\tau \\
&\qquad \le 
\frac 12 \nor{(1+A)^{\frac s2} Q_H v_0}{L^2}^2
+ \frac{1}{c_0\al} \int_0^t e^{c_0\al\tau} \nor{(1+A)^{\frac {s-1}2} Q_H F(\tau)}{L^2}^2 d\tau,
\end{align*}
which, combined with Lemma \ref{L4}, implies \eqref{e7.2}.
\end{proof}
\subsection{Low-frequency estimate}
Next we derive the estimate for 
$e^{t(-\al+i)A}Q_L$. 
The key ingredient is the precise asymptotics of $\la_0(\xi)$ and $\phi_0(\xi,\cdot)$ obtained in Proposition \ref{P3}. 
Combined with explicit representation of the integral kernel for $e^{t(-\al+i)A}Q_L$ (see Lemma \ref{Lx5}), we derive $L^1(\R^d)\to L^p(\R^d)$ estimates for $2\le p\le \infty$. 
This scheme is based on \cites{MR2601067,MR2788923}. 
\begin{lem}[Fundamental solution]\label{Lx5}
For $v_0\in L^1(\R^d)$, we have
\[
e^{t(-\al+i)A} Q_L v_0 = \int_{\R^d} G_L(t,x,y) v_0(y) dy
\]
where
\[
G_L(t,x,y)
=
\frac{1}{(2\pi)^{d-1}}
\int_{\T^*}
e^{i\xi\cdot (x-y)}\chi(\xi)
e^{t(-\al+i)\la_0(\xi)}
\phi_0(\xi,x_1) \ovl{ \phi_0 (\xi, y_1)} d\xi.
\]
\end{lem}
\begin{proof}
The proof is essentially given by Proposition 6.2 in \cite{MR2601067}. 
We provide an explicit proof here for the sake of clarity. 
We may suppose $v_0\in \boS(\R^d)$; 
for general 
$v_0\in L^1(\R^d)$, the result can be generalized via a density argument. 
First, by definition of $Q_L$, we have
\begin{align*}
e^{t(-\al+i)A} Q_L v_0
& =
\int_{\T^*} e^{i\xi\cdot x} \chi(\xi) e^{t(-\al +i)\la_0 (\xi)} P_0(\xi)  \BF [v_0] (\xi,x_1) d\xi \\
&
=
\int_{\T^*} e^{i\xi\cdot x} \chi(\xi)
e^{t(-\al +i)\la_0 (\xi)} \\
&\hspace{40pt} \cdot\left[\int_0^{2\pi} \BF [v_0] (\xi,z_1) \ovl{\phi_0 (\xi, z_1)} dz_1\right] \phi_0 (\xi,x_1)  d\xi.
\end{align*}
Here, for $k\in\Z$, we denote the $k$-th Fourier coefficient of $\phi_0(\xi,\cdot)$ by 
\[\widehat{\phi_0} (\xi,k) = \frac 1{2\pi} \int_0^{2\pi} e^{-ikz_1} \phi_0(\xi,z_1) dz_1.\] 
Then, we have
\begin{align*}
      \int_0^{2\pi} \BF [v_0] (\xi,z_1) \ovl{\phi_0 (\xi, z_1)} dz_1
       & =
      \frac 1{(2\pi)^d}
      \int_0^{2\pi} \sum_{k\in \Z} e^{ikz_1} \boF [v_0] (\xi_1 + k ,\xi')
      \ovl{\phi_0 (\xi, z_1)} dz_1                                                                       \\
       & = \frac{1}{(2\pi)^{d-1}}
      \sum_{k\in \Z} \boF [v_0] (\xi_1 + k ,\xi') \ovl{ \widehat{\phi_0} (\xi, k)}                      \\
       & = \frac{1}{(2\pi)^{d-1}}
      \sum_{k\in \Z} \int_{\R^d} e^{-i\xi\cdot y} e^{-iky_1} v_0 (y) dy \ovl{ \widehat{\phi_0} (\xi, k)} \\
       & = \frac{1}{(2\pi)^{d-1}}
      \int_{\R^d} e^{-i\xi\cdot y}
      v_0 (y) \ovl{ \phi_0 (\xi, y_1)} dy.
\end{align*}
Hence, we obtain
\begin{align*}
      & e^{t(-\al +i)A}Q_L v_0 \\
      & =
      \frac{1}{(2\pi)^{d-1}}
      \int_{\T^*} e^{i\xi\cdot x} \chi(\xi)
      e^{t(-\al+ i)\la_0(\xi)}
      \left[
        \int_{\R^d} e^{-i\xi\cdot y}
        v_0 (y) \ovl{ \phi_0 (\xi, y_1)} dy \right]
      \phi_0(\xi,x_1) d\xi \\
                     & =
                     \frac{1}{(2\pi)^{d-1}}
      \int_{\R^d}
      \int_{\T^*}
      e^{i\xi\cdot (x-y)}
      \chi(\xi)
      e^{t(-\al+ i)\la_0(\xi)}
      \phi_0(\xi,x_1) \ovl{ \phi_0 (\xi, y_1)} v_0(y) d\xi dy,
\end{align*}
which completes the proof.
\end{proof}
\begin{lem}\label{Lx6}
The following estimates hold.
\begin{gather}\label{e7.3}
\sup_{y\in \R^d} \nor{G_L(t,\cdot ,y)}{L^p_x (\R^d)}
\le C(1+\al t)^{-\frac d2\left(1-\frac 1p\right)},
\\
\label{e7.4}
\sup_{y\in \R^d} \nor{\nab_{x,y}^k G_L(t,\cdot ,y)}{L^p_x (\R^d)}
\le C_k (1+\al t)^{-\frac d2\left(1-\frac 1p\right) -\frac 12},\qquad k\ge 1.
\end{gather}
\end{lem}
\begin{proof} 
We begin with the proof of \eqref{e7.3}. 
It suffices to consider $p=2$ and $p=\infty$, since the general case follows by interpolation. 
We first observe that 
Proposition \ref{P3} implies $|\phi_0(\xi,x_1)| \le C$ uniformly in $|\xi|_*\le \del_0$ and $x_1\in\T$. Hence, \eqref{E3.12} yields
\[
    |G_L(t,x,y)| \le C\int_{\T^*}
    \chi(\xi) e^{-\al t\la_0(\xi)}
    |\phi_0 (\xi,x_1) | |\phi_0(\xi,y_1)| d\xi
    \le
    C \int_{\T^*} \chi(\xi) e^{-\te_0 \al t|\xi|_*^2} d\xi.
\]
If $t\le \al^{-1}$, then
\[
    \int_{\T^*} \chi(\xi) e^{-\te_0 \al t|\xi|_*^2} d\xi
    \le \int_{\T^*} \chi(\xi) d\xi
    \le C \le \frac{C}{(1+\al t)^{\frac d2}}.
\]
If $t> \al^{-1}$, then
\[
    \int_{\T^*} \chi(\xi) e^{-\te_0 \al t|\xi|_*^2} \le
    \int_{\T^*} e^{-\te_0 \al t|\xi|_*^2} d\xi
    \le C (\al t)^{-\frac d2} \le \frac{C}{(1+\al t)^{\frac d2}}.
\]
Thus, combining these gives \eqref{e7.3} with $p=\infty$. 
Next, we note that
\[
    G_L(t,x,y) = \BF^{-1} [U_{y}] (t,x)
\]
with
\[
    U_y(\xi,x_1)
    = \frac{1}{(2\pi)^{d-1}}
    e^{-i\xi\cdot y} \chi(\xi) e^{t(-\al +i)\la_0(\xi)} \phi_0(\xi, x_1) \ovl{\phi_0(\xi, y_1)}.
\]
Then we have
    \begin{align*}
      \nor{G_L(t,x,y)}{L^2_x(\R^d)}
       & = C \nor{ \chi(\xi) e^{t(-\al +i)\la_0(\xi)} \phi_0(\xi, x_1) \ovl{\phi_0(\xi, y_1)} }{L^2_\xi L^2_{x_1}} \\
       & \le C \nor{\chi(\xi) e^{-\te_0 \al t |\xi|_*^2}}{L^2_\xi} \nor{\phi_0(\xi,x_1)}{L^\infty_\xi L^2_{x_1}}     \\
       & \le C (1+\al t)^{-\frac d4},
    \end{align*}
which concludes \eqref{e7.3} in the case $p=2$. 
For \eqref{e7.4}, we only give a proof for 
$\nab_x^k G(t,x,y)$ with $k\in\N$, while 
the remaining cases 
can be handled in the similar way. 
In addition, we restrict the argument to the case $d=2$, while the case $d=3$ follows similarly. 
For $k_1,k_2\in\Z_{\ge 0}$ with $k=k_1+k_2\ge 1$, direct computation gives
\begin{align*}
&\rd_{x_1}^{k_1} \rd_{x_2}^{k_2} G_L(t,x,y)\\
&=
\int_{\T^*}
e^{i\xi\cdot (x-y)} \chi(\xi) e^{t(-\al +i)\la_0(\xi)} (i\xi_1 + \rd_{x_1})^{k_1} (i\xi_2)^{k_2} \phi_0(\xi,x_1)
\ovl{\phi_0(\xi,y_1)} d\xi.
\end{align*}
Here, \eqref{E3.14} yields
\begin{equation}\label{e7.5}
    \rd_{x_1}^{l} \phi_0(\xi,x_1) =
    \begin{cases}
    \frac 1{\sqrt{2\pi}} \xi_2 \cos \left(x_1 + \frac{l\pi}{2} \right) + \rd_{x_1}^{l} R_0(\xi, x_1), & d=2,\\
    \frac 1{\sqrt{2\pi}} \xi_2 \cos \left(x_1 + \frac{l\pi}{2} \right) +
    \frac 1{\sqrt{2\pi}} \xi_3 
    \sin \left(x_1 + \frac{l\pi}{2} \right)
    +
    \rd_{x_1}^{l} R_0(\xi, x_1), & d=3
    \end{cases}
\end{equation}
for $l\in \N$. 
Therefore, 
we have
\begin{equation}\label{e7.55}
    |(i\xi_1 +\rd_{x_1})^{k_1} (i\xi_2)^{k_2} \phi_0(\xi,x_1)| \le C |\xi|_*.
 \end{equation}
Hence, \eqref{e7.4} is concluded in the same way as \eqref{e7.3}, by noting that
\[
\nor{\chi(\xi) |\xi|_* e^{-\te_0 \al t|\xi|_*^2} }{L^q_\xi} \le 
C (1+\al t)^{-\frac 12 - \frac d{2q} }\qquad \text{for } q=1, 2.
\]
\end{proof}
Finally, we claim the main estimates of this subsection.
\begin{prop}[Low-frequency estimate]\label{P7.20}
For $2\le p \le\infty$ and $v_0\in L^1(\R^d)$, we have
\begin{gather}
\label{e7.6}
\nor{e^{t(-\al+i)A} Q_L v_0}{L^p(\R^d)} \le C (1+\al t)^{-\frac d2 \left(1 - \frac 1p\right)} \nor{v_0}{L^1(\R^d)},
\\
\label{e7.7}
\begin{multlined}
      \nor{\nab^k e^{t(-\al+i)A} Q_L v_0}{L^p(\R^d)}
      + \nor{e^{t(-\al+i)A} Q_L \nab^k v_0}{L^p(\R^d)}\\
      \le C_k (1+\al t)^{-\frac d2 \left(1 - \frac 1p\right) -\frac 12} \nor{v_0}{L^1(\R^d)},\qquad (k\ge 1).
\end{multlined}
\end{gather}
\end{prop}
\begin{proof}
For $e^{t(-\al+i)A} Q_L v_0$ and $\nab^k e^{t(-\al+i)A} Q_L v_0$, the H\"older inequality implies
  \begin{align*}
    \nor{e^{t(-\al+i)A} Q_L v_0}{L^p(\R^d)}
    &\le
    \int_{\R^d}
    \nor{G_L(t,\cdot,y)}{L^p_x(\R^d)} |v_0(y)| dy \\
    &\le
    C (1+\al t)^{-\frac d2 \left(1- \frac 1p\right)} \nor{v_0}{L^1(\R^d)},\\
    \nor{\nab^k e^{t(-\al+i)A} Q_L v_0}{L^p(\R^d)}
    &\le
    \int_{\R^d}
    \nor{\nab_x^k G_L(t,\cdot,y)}{L^p_x(\R^d)} |v_0(y)| dy\\
    &\le
    C (1+\al t)^{-\frac d2 \left(1- \frac 1p\right)-\frac 12} \nor{v_0}{L^1(\R^d)}.
\end{align*}
Next, if $v_0\in \boS(\R^d)$, integration by part yields
\[
e^{t(-\al+i)A} Q_L \nab^k v_0 =
\int_{\R^d} G_L(t,x,y) \nab_y^k v_0 (y) dy
=
\int_{\R^d} (-\nab_y)^k G_L(t,x,y) v_0 (y) dy.
\]
Hence the estimate of $e^{t(-\al+i)A} Q_L \nab^k v_0$ for $v_0\in \boS (\R^d)$ follows in the same manner as above. 
By a density argument, 
$v_0$ is generalized into any $L^1$-function, completing the proof.
\end{proof}
\begin{remark}
In view of \eqref{e7.5} and \eqref{e7.55}, 
one can observe that differentiating in $x_1$ will not improve the decay rate further than $\frac 12$. 
This is a different phenomenon from the standard heat equation. 
\end{remark}
\subsection{The case $d=1$}
Note that $A$ coincides with the standard Laplacian when $d=1$. Thus, the similar estimates are valid if we define the frequency cutoff via Fourier-wave numbers. 
Namely, when $d=1$, we substitute the definition of $Q_L$, $Q_H$ by
\begin{equation}\label{e1:1}
Q_L
= \boF^{-1} \mathbbm{1}_{|\xi|\le 1} \boF ,\qquad 
Q_H  = 
1-Q_L = \boF^{-1} \mathbbm{1}_{|\xi|> 1} \boF.
\end{equation}
Then for the high-frequency part, the same estimate as \eqref{e7.2} holds true. 
For the low-frequency part, the integral kernel for $e^{t(-\al+i)A}Q_L$ is represented as
\[
G_L (t,x,y) = 
\frac{1}{2\pi} \int_{\R} e^{i\xi(x-y)} 
\mathbbm{1}_{|\xi|\le 1}(\xi) e^{t(-\al+i)\xi^2} 
d\xi.
\]
Hence, the similar argument as in the proof of Lemma \ref{Lx6} and Proposition \ref{P7.20} yields 
\begin{equation}
\begin{multlined}
\nor{\rd_1^k e^{t(-\al+i)A} Q_L v_0}{L^p(\R)} + 
\nor{e^{t(-\al+i)A} Q_L \rd_1^k v_0}{L^p(\R)}
\\
\le C_k (1+\al t)^{-\frac 12\left(1-\frac 1p\right)- \frac {k}{2} }
\nor{v_0}{L^1(\R)}
\end{multlined}
\end{equation}
for $2\le p\le \infty$, $k\in \Z_{\ge0}$, and $v_0\in L^1(\R)$.
%
%
%
%
%
%
%
%
%
%
\section{Nonlinear estimate}\label{S5}
In this section, we give the estimate for \eqref{E2.4} to conclude Theorem \ref{T1}. The concerning nonlinear problem is rewritten as 
\begin{equation}\label{eb7.1}
  \left\{
  \begin{aligned}
      \rd_t u &= (-\al +i)A u + N(u), \qquad & (t,x)\in (0,T)\times \R^d, \\
      u(0,x) &= u_0(x),\qquad & x\in \R^d,
  \end{aligned}
  \right.
\end{equation}
where $T>0$ and
\[
N(u) = \al N_1(u, \nab u) + N_2(u,\nab u, \nab^2 u)
\]
as $N_1,N_2$ defined in Section \ref{S2}. 
Following the idea of \cite{MR3243362}, 
we decompose $u$ into low- and high-frequency parts, 
and then estimate each part separately. 
Here, the corresponding decomposition is carried out 
by using the operators $Q_L$, $Q_H$ defined in 
\eqref{ea6.1}, \eqref{ea6.2}. Namely, we set
\[
u = u_L + u_H,\qquad u_L = Q_L u,\qquad u_H = Q_H u.
\]
Then for $t\in [0,T_{\max})$, define
\begin{align*}
   & M (T) = M_L(T) + M_H(T),                                   \\
   & M_L(T)
  =
\sup_{t\in [0,T]} \left( (1+\al t)^{\frac d4 } \nor{u_L(t) }{L^2} +
  (1+\al t)^{\frac d4+\frac 12} \nor{\nab u_L (t)}{H^{s}} \right.\\
  &\hspace{100pt} 
  \left.+ (1+\al t)^{\frac d2} \nor{u_L(t)}{L^{\infty}}
  +
  (1+\al t)^{\frac d2+\frac 12} \nor{\nab u_L(t)}{W^{s,\infty}}
  \right), \\
   & M_H(T) =
\sup_{t\in [0,T]} \left( 
(1+\al t)^{\frac d2+1} \nor{u_H(t)}{H^s} \right.\\
&\hspace{80pt}  \left.+ \al^{\frac 12} (1+\al t)^{\frac d2+1} \left[
\int_0^t e^{-c_0 \al (t-\tau)} \nor{u_H(\tau)}{H^{s+1}}^2 d\tau
\right]^{\frac 12}\right) 
\end{align*}
where $c_0$ is the constant as in Proposition \ref{P7.20}. 
We remark that the idea of the definition of $M_H$ comes from \cite{MR4316931}. 
The key estimate is the following:
\begin{prop}[Nonlinear estimate]\label{P7.22}
Let $T>0$ be a number and $s\ge 2$ be an integer. 
Let 
$u\in C([0,T]\colon H^s(\R^d))\cap L^2(0,T\colon H^{s+1}(\R^d))$ 
be a solution to \eqref{eb7.1} with $u_0\in L^1(\R^d)\cap H^s(\R^d)$. 
Suppose that $\sup_{t\in [0,T],x\in \R^d} |u(t,x)|\le \frac 12$. 
Then 
there exist a constant $C_0>0$ and an integer $K\ge 2$ such that
\begin{equation}\label{eb7.11}
    M(t) \le C_0 \nor{u_0}{L^1\cap H^{s}} + C_0 (1+\al^{-1}) (M(t)^2+M(t)^K) 
\end{equation}
for all $t\in[0,T]$.
\end{prop}
We first show that Proposition \ref{P7.22} implies Theorem \ref{T1}.
\begin{proof}[Proof of Theorem \ref{T1}]
The proof is based on the bootstrap argument. 
Let $\eps_0>0$ be a small number determined later, and let $\bom_0\in L^1\cap H^s$ with $\nor{\bom_0}{L^1\cap H^s}<\eps_0$. 
Also let $\bom \in C([0,T_{\max})\colon H^s (\R^d))\cap L^2_{\loc}(0,T_{\max}\colon H^{s+1}(\R^d))$ 
be the unique solution to \eqref{E1.7} as in Proposition \ref{P1}. 
We set $u=u(t,x)$ to be the associated function to $\bom(t,x)$ defined via \eqref{E2.1}, and write $u_0(x)=u(0,x)$.  
Then, by \eqref{E2.275} and \eqref{e2:2}, there exists $C_1>0$ such that $\nor{u_0}{L^1\cap H^s}\le C_1 \nor{\bom_0}{L^1\cap H^s} \le C_1\eps_0$. 
Now, let $R>0$ be a number satisfying 
$C_1R>4$. 
Then, we set
\[
T^* = \sup \{ T\in [0,T_{\max}) \ |\ 
\sup_{t\in [0,T],x\in \R^d} |u(t,x)|\le \frac 12,\ 
M(T) \le R C_0 C_1 \eps_0\}, 
\]
where $C_0$ is the constant as in Proposition \ref{P7.22}. 
We first observe that $T_*>0$ holds if we choose $R$ sufficiently large. 
By the regularity conditions of $\bom$, we have $\nor{\bom}{L^\infty_TH^s_x} + \nor{\bom}{L^2_TH^{s+1}_x} \lesssim \nor{\bom_0}{H^s}=\eps_0$ for $T\ll 1$, 
which particularly guarantees $\nor{u}{L^\infty_TL^\infty_x}\le \frac 12$ by the Sobolev embedding. 
Using Lemmas \ref{L4} and \ref{L3}, together with \eqref{e2:2}, we obtain
\[
M(T)\lesssim \nor{u}{L^\infty_TH^s_x} + 
\nor{u}{L^2_T H^{s+1}_x}
\lesssim \nor{\bom}{L^\infty_TH^s_x} + \nor{\bom}{L^2_TH^{s+1}_x}
\lesssim\eps_0
\]
for $T\ll 1$. 
Therefore, if $R$ is sufficiently large, we have $M(T)\le R C_0 C_2 \eps_0$ for $T\ll 1$, which implies $T_*>0$. \par
Now we show that $T_{\max}=\infty$. 
By virtue of \eqref{E2.3} and \eqref{E1.8}, 
it suffices to show $T^*=T_{\max}$. 
First, since $s\ge 2$, the Sobolev embedding implies $u\in C_tL^\infty_x$. 
Moreover, we have
\begin{align*}
\nor{u(t)}{L^\infty} 
\le \nor{u_L(t)}{L^\infty} + \nor{u_H(t)}{L^\infty}
&\lesssim (\nor{u_L(t)}{L^\infty} + \nor{u_H(t)}{H^s}) \\
&\lesssim (1+\al t)^{-\frac d2} M(t) \lesssim \eps_0 
\end{align*}
for $t\in [0,T^*)$. 
Hence, if $\eps_0$ is sufficiently small, 
we obtain $\nor{u}{L^\infty_{T^*}L^\infty_x}\le \frac 14$. 
Next, by the argument in Section \ref{S2}, it follows that $u$ solves 
\eqref{eb7.1}, 
and thus we may apply Proposition \ref{P7.22}. Therefore, for $t\in [0,T^*)$, we have
\[
M(t) \le 
C_0 \nor{u_0}{L^1\cap  H^{s}} + C_2 (1+\al^{-1}) \eps_0 M(t)
\]
with a constant $C_2>0$ independent of $\eps$. 
Taking $\eps_0 < (2C_2 (1+\al^{-1}))^{-1}$, we obtain
\begin{equation}\label{Ex5.3}
M(t) \le 
2 C_0 \nor{u_0}{L^1\cap  H^{s}} \le 
\frac 12 R C_0 C_1\eps_0
\end{equation}
for all $t\in [0,T^*)$. 
Therefore, the bootstrap argument yields $T^*=T_{\max}$, concluding $T_{\max}=\infty$.\par
Finally, by the Sobolev embedding, we have
\begin{align*}
&(1+\al t)^{\frac d4} \nor{u(t)}{H^s} + 
(1+\al t)^{\frac d2} \nor{u(t)}{W^{\lfloor s-\frac {d+1}2 \rfloor,\infty}} \\
&\lesssim 
(1+\al t)^{\frac d4} \nor{u_L(t)}{H^s} 
+(1+\al t)^{\frac d2} \nor{u_L(t)}{W^{s,\infty}} 
+(1+\al t)^{\frac d2+1} \nor{u_H(t)}{H^s}  \\
&\lesssim  M(t)
\end{align*}
for $t>0$. 
Therefore, 
by \eqref{Ex5.3} and \eqref{E2.275}--\eqref{e2:2}, we obtain
\begin{align*}
&(1+\al t)^{\frac d4} \nor{\bom (t)}{H^s} 
+ (1+\al t)^{\frac d2} \nor{\bom (t)}{W^{\lfloor s-\frac {d+1}2 \rfloor,\infty}}\\
&\lesssim  
(1+\al t)^{\frac d4} (\nor{u(t)}{H^s} +\nor{u(t)}{H^s}^s)\\
&\qquad + 
(1+\al t)^{\frac d2} (\nor{u(t)}{W^{\lfloor s-\frac {d+1}2 \rfloor,\infty}}+ \nor{u(t)}{W^{\lfloor s-\frac {d+1}2 \rfloor,\infty}}^s)\\
&\lesssim  M(t) + M(t)^s \\
&\lesssim  \nor{u_0}{L^1\cap H^s} + 
\nor{u_0}{L^1\cap H^s}^s  
\lesssim \nor{\bom_0}{L^1\cap H^s} + \nor{\bom_0}{L^1\cap H^s}^s
\lesssim \nor{\bom_0}{L^1\cap H^s}
\end{align*}
for $t>0$, which concludes \eqref{E1.9}.
\end{proof}
\subsection{Low-frequency bound}
We begin by estimating $M_L(t)$. 
We take $Q_L$ on both sides of \eqref{eb7.1}, then 
\[
u_L(t) = e^{t(-\al+i)A} Q_L u_0 + \int_0^t e^{(t-\tau)(-\al+i)A} Q_L N(u(\tau)) d\tau.
\]
Applying Proposition \ref{P7.20}, we obtain
\begin{multline*}
(1+\al t)^{\frac d4} \nor{u_L}{L^2} 
+
(1+\al t)^{\frac d4+\frac 12} \nor{\nab u_L}{H^{s}} \\
+
(1+\al t)^{\frac d2} \nor{u_L}{L^{\infty}}
+
(1+\al t)^{\frac d2+\frac 12} \nor{\nab u_L}{W^{s,\infty}}\\
\le 
C \nor{u_0}{L^1}
+
C\sum_{a=\frac d4,\frac d4+\frac 12,\frac d2,\frac d2+\frac 12} 
(1+\al t)^{a} \int_0^t (1+\al(t-\tau))^{-a} 
\nor{N(u(\tau))}{L^1} d\tau.
\end{multline*}
Here, we claim that
\begin{equation}\label{E5.3}
\nor{N(u (\tau)}{L^1} \le C (1+\al) (1+\al \tau)^{-\frac d2-\frac 12} M(\tau)^2,\qquad \tau\in [0,T].
\end{equation}
Suppose that \eqref{E5.3} is true. Then, we have
\begin{align*}
&M(t) 
\le C\nor{u_0}{L^1} \\
&\quad + C (1+\al) M(t)^2 
\sum_{a=\frac d4,\frac d4+\frac 12,\frac d2,\frac d2+\frac 12}
(1+\al t)^{a}
\int_0^t (1+\al (t-\tau))^{-a} (1+\al \tau)^{-\frac d2 -\frac12} d\tau.
\end{align*}
For $a=\frac d4$, $\frac d4 +\frac 12$, $\frac d2$ or $\frac d2 +\frac 12$, direct computation yields
\begin{align*}
   & \int_0^t (1+\al (t-\tau))^{-a} (1+\al \tau)^{-\frac d2-\frac 12} ds            \\
   & = \int_0^{\frac t2} (1+\al (t-\tau))^{-a} (1+\al \tau)^{-\frac d2-\frac 12} d\tau
  + \int_{\frac t2}^{t} (1+\al (t-\tau))^{-a} (1+\al \tau)^{-\frac d2-\frac 12} d\tau  \\
   & \le C (1+\al t)^{-a} \int_{0}^{\frac t2} (1+\al \tau)^{-\frac d2-\frac 12} d\tau 
+ C (1+\al t)^{-\frac d2-\frac 12} \int_{\frac t2}^{t} (1+\al (t-\tau))^{-a}  d\tau \\
&\le C \al^{-1} (1+\al t)^{-\min \{ a,\frac d2+\frac 12 \}}
= C\al^{-1}(1+\al t)^{-a}.
\end{align*}
Hence, we obtain
\begin{equation}\label{eb7.12}
M_L(t) \le C \nor{u_0}{L^1 \cap H^s} + C (1+\al^{-1}) M(t)^2,
\end{equation}
which is exactly the desired estimate for $M_L$.\par
Now we prove \eqref{E5.3}. 
By the assumption $|u|\le \frac 12$ and \eqref{E2.25}, 
we have
\begin{align*}
&\nor{N_1(u,\nab u)}{L^1} \\
&\lesssim 
\nor{u \nab u}{L^1} + \nor{ |\nab u|^2}{L^1} \\
&\le 
\nor{u}{L^2} \nor{\nab u}{L^2} + \nor{ \nab u}{L^2}^2\\ 
&\le
\left( \nor{u_L}{L^2} + \nor{u_H}{L^2} \right) 
\left( \nor{\nab u_L}{L^2} + \nor{\nab u_H}{L^2} \right) 
+
\left( \nor{\nab u_L}{L^2} + \nor{\nab u_H}{L^2} \right)^2 \\
&\lesssim 
\left( (1+\al t)^{-\frac d4} M_L(t) + (1+\al t)^{-\frac d2 -1} M_H(t) \right) \\
&\hspace{60pt} \cdot \left( (1+\al t)^{-\frac d4 - \frac 12} M_L(t) + (1+\al t)^{-\frac d2 -1} M_H(t) \right)\\
&\qquad + 
\left( (1+\al t)^{-\frac d4 -\frac 12} M_L(t) + (1+\al t)^{-\frac d2 -1} M_H(t) \right)^2\\
&\lesssim  
(1+\al t)^{-\frac d2-\frac 12} M(t)^2 ,
\end{align*}
where the implicit constants are independent of $t$ and $\al$. 
For $N_2$, the terms involving up to first-order derivatives are controlled in the same way as above. 
By noting \eqref{e2:3}, 
we estimate the second-order derivative terms as
\begin{align*}
\nor{z\Del u}{L^1} &\lesssim 
\nor{u}{L^2} \nor{\Del u}{L^2} \\
&\le \left( \nor{u_L}{L^2} + \nor{u_H}{L^2} \right) \left( \nor{\Del u_L}{L^2} + \nor{\Del u_H}{L^2} \right) \\
&\lesssim
\left( (1+\al t)^{-\frac d4} M_L(t) + (1+\al t)^{-\frac d2 -1} M_H(t) \right) \\
&\hspace{60pt}
\cdot \left( (1+\al t)^{-\frac d4-\frac 12} M_L(t) + (1+\al t)^{-\frac d2 -1} M_H(t) \right)\\
&\lesssim (1+\al t)^{-\frac d2-\frac 12} M(t)^2,\\
\nor{u\Del z}{L^1} 
&\lesssim 
\nor{u}{L^2} \nor{\nab^2 u}{L^2} + \nor{\nab u}{L^2}^2
\lesssim (1+\al t)^{-\frac d2-\frac 12} M(t)^2.
\end{align*}
Hence \eqref{E5.3} follows.
\subsection{High-frequency bound}
Next, we move on to the estimate of $M_H(t)$. 
Applying $Q_H$ on both sides of \eqref{eb7.1}, we have
\[
u_H(t) = e^{t(-\al+i)A} Q_H u_0 + \int_0^t e^{(t-\tau)(-\al+i)A} Q_H N(u(\tau)) d\tau.
\]
Hence, \eqref{e7.2} yields
\begin{align*}
   & \nor{u_H(t)}{H^{s}} +\al^{\frac 12} \left[ \int_0^t e^{-c_0 \al (t-\tau)} \nor{u_H(\tau)}{H^{s+1}}^2 d\tau \right]^{\frac 12} \\
   & \le
C e^{-\frac{c_0}{2} \al t}  \nor{ Q_H u_0}{H^s} 
+ C \al^{-\frac 12} \left[
\int_0^t e^{-c_0 \al (t-\tau)} \nor{Q_H N(u(\tau)) d\tau}{H^{s-1}}^2 d\tau
\right]^{\frac 12} .
\end{align*}
Now we claim the following:
\begin{lem}\label{L7.12}
Let $u$ be a solution satisfying the condition in Proposition \ref{P7.22}. 
Then, there exist a number $C>0$ and an integer $K\ge 2$, both of which are independent of $\al$, such that
\begin{equation}\label{Ex5.45}
\begin{aligned}
\nor{Q_H N(u(\tau))}{H^{s-1}} 
&\le C (1+\al) (1+\al \tau)^{-\frac d2-1} (M(\tau)^2 + M(\tau)^K) \\
&\qquad + C (1+\al) (M(\tau) + M(\tau)^{K-1}) \nor{u_H(\tau)}{H^{s+1}}
\end{aligned}
\end{equation}
for $\tau\in [0,T]$.
\end{lem}
Suppose that the claim is true. Then, we have 
\begin{align*}
& (1+\al t)^{\frac d2+1} \nor{u_H(\tau)}{H^s} 
+\al^{\frac 12}  (1+\al t)^{\frac d2+1} \left[ \int_0^t e^{-c_0 \al (t-\tau)} \nor{u_H(\tau)}{H^{s+1}}^2 d\tau \right]^{\frac 12} \\
 &\lesssim e^{-\frac {c_0}2 \al t}  (1+\al t)^{\frac d2+1} \nor{u_0}{H^s} \\
&\quad +  (1+\al) \al^{-\frac 12}
 (1+\al t)^{\frac d2+1} (M(t)^2 +M(t)^K) 
\left[\int_0^t e^{-c_0 \al (t-\tau)} (1+\al \tau)^{-d-2}  d\tau\right]^{\frac 12} \\
&\quad 
+ (1+\al) \al^{-\frac 12}  (1+\al t)^{\frac d2+1}
 (M(t) +M(t)^{K-1}) 
\left[\int_0^t e^{-c_0 \al (t-\tau)} \nor{u_H(\tau)}{H^{s+1}}^2  d\tau\right]^{\frac 12}\\
&\lesssim 
e^{-\frac {c_0}2 \al t} (1+\al t)^{\frac d2+1} \nor{u_0}{H^s} \\
&\quad +  (1+\al^{-1}) (M(t)^2+M(t)^K) 
 +  (1+\al^{-1}) (M(t)+M(t)^{K-1}) M_H(t),
\end{align*}
where the estimate in the second term follows from 
\begin{align*}
&\int_0^t e^{-c_0 \al (t-\tau)} (1+\al \tau)^{-d-2} d\tau \\
&\lesssim 
 e^{-\frac 12 c_0 \al t}
\int_0^{\frac t2} (1+\al \tau)^{-d-2} d\tau 
+  (1+\al t)^{-d-2}
\int_{\frac t2}^t e^{-c_0 \al (t-\tau)} d\tau \\
&\lesssim \al^{-1} (1+\al t)^{-d-2}.
\end{align*}
Hence, we obtain
\[
M_H(t) \le C \nor{u_0}{L^1 \cap H^s} + C (1+\al^{-1}) (M(t)^2 + M(t)^K),
\]
which, combined with \eqref{eb7.12}, concludes \eqref{eb7.11} as desired.
\begin{proof}[Proof of Lemma \ref{L7.12}]
First of all, \eqref{Ea4.4} implies
\[
\nor{Q_H N(u(\tau))}{H^{s-1}} = 
\nor{N(u(\tau)) - Q_L N(u(\tau))}{H^{s-1}}
\lesssim \nor{N(u(\tau))}{H^{s-1}}.
\]
Before the nonlinear estimate, we note that $\rd^\gam z$ for a multi-index $\gam\in (\Z_{\ge0})^d$ with $|\gam|\ge 1$ can be written as a linear combination of the terms of the form
\[
(1-|u|^2)^{-\frac 12-k}  \rd^{\gam_1} u \cdots\rd^{\gam_{l'}} u \rd^{\gam_{l'+1}} \ovl{u} \cdots \rd^{\gam_l} \ovl{u}
\]
with $l\ge2$, $0\le l'\le l$, $k\ge 0$, $\gam_j\in (\Z_{\ge0})^d$ with $\gam_1+\cdots +\gam_l = \gam$. 
In particular, the derivatives of $z$ 
can be bounded in terms of $u$ in higher order.\par 
Now, we divide our argument into the cases $s\ge 3$ and $s=2$. 
We begin by considering the case $s\ge 3$. 
Among $N(u(\tau))$, we first observe the terms including second-order derivatives of $u$ or $z$. 
Using the condition $|u| \le \frac 12$, together with \eqref{E2.2}, the Leibniz rule, and the Sobolev embedding $H^2\subset L^\infty$, we have
\begin{equation}\label{Ex5.5}
\begin{aligned}
&\nor{z\Del u}{H^{s-1}} 
\lesssim \nor{z\Del u}{L^2} + \nor{\nab^{s-1}(z\Del u)}{L^2} \\
&\lesssim 
\nor{z}{L^\infty} \nor{\Del u}{L^2} 
+\nor{z}{L^\infty} \nor{\nab^{s-1}\Del u}{L^2}
+\sum_{j=1}^{s-1} \nor{\nab^j z}{L^4} 
\nor{\nab^{s-1-j}\Del u}{L^4}\\
&\lesssim 
\nor{u}{L^\infty}^2 \nor{\Del u}{L^2} 
+ 
(\nor{u}{L^\infty}^2 + \nor{\nab z}{W^{s-2,4}})
(\nor{\nab u_L}{H^s} + \nor{\nab u_H}{H^s} )\\
&\lesssim (1+\al \tau)^{-\frac d2-1} (M(\tau)^2 + M(\tau)^K) + 
(M(\tau) + M(\tau)^{K-1}) \nor{u_H}{H^{s+1}}
\end{aligned}
\end{equation}
for some $K\ge 2$. 
The term $iu\Del z$ is also controlled in a similar way. 
By means of the Leibniz rule and the Sobolev embedding, 
the other terms including up to first-order derivatives of $u$ and $z$ can be bounded by 
\begin{equation}\label{Ex5.6}
\begin{aligned}
&C (1+\al)
\left[ (\nor{u}{L^\infty} + \nor{\nab u}{H^{s-1}})^2 + 
(\nor{u}{L^\infty} + \nor{\nab u}{H^{s-1}})^K \right]  \\
&\le C (1+\al) (1+\al\tau)^{-\frac d2-1} (M(\tau)^2 + M(\tau)^K).
\end{aligned}
\end{equation}
Hence, we obtain \eqref{Ex5.45} for $s\ge 3$.\par
When $s=2$, we estimate the term $|\nab u|^2u$ as
\begin{align*}
\nor{|\nab u|^2 u}{H^1} &\lesssim 
\nor{|\nab u|^2 u}{L^2}
+\nor{\nab^2 u}{L^2} \nor{\nab u}{L^\infty} \nor{u}{L^\infty} + 
\nor{|\nab u|^3}{L^2} \\
&\lesssim
\nor{\nab u}{H^1}^2 + 
\nor{\nab^2 u}{L^2} (\nor{\nab u_L}{L^\infty} + \nor{\nab u_H}{H^2}) 
+ \nor{\nab u}{H^1}^3\\
&\lesssim 
(1+\al t)^{-\frac d2 -1} (M(\tau)^2 + M(\tau)^3) 
+ M(\tau) \nor{u_H}{H^3}. 
\end{align*}
We apply similar estimates to $|\nab z|^2 u$, $iu\Del z$ as well. 
All the remaining terms can be bounded in the same way as \eqref{Ex5.5} or \eqref{Ex5.6}. Hence, we obtain \eqref{Ex5.45} as desired.
\end{proof}
\appendix
\section{Local well-posedness}\label{SA}
In this section, we give a proof of Proposition \ref{P1}. What we consider is the initial-value problem for \eqref{E1.9}; namely,
\begin{equation}\label{ec7.2}
\left\{
\begin{aligned}
\rd_t \bom &= 
\al \Del\bom - (\bh + \bom)\times \Del \bom + N(\bom), && (t,x) \in [0,\infty)\times \R^d,\\
\bom(0,x) &= \bom_0(x), && x\in \R^d
\end{aligned}
\right.
\end{equation}
for given $\bom_0\in H^s(\R^d)$, 
where  
\[
N(\bom) = 
-\al \nab\times \bom + (\bh + \bom)\times (\nab \times \bom) + \al \Gam(\bom)
\]
with $\Gam(\bom)$ defined in \eqref{E1.75}. 
The proof is based on the standard energy method 
for the principal heat part $\rd_t \bom = \al \Del\bom$, 
while all the other terms are handled perturbatively. 
As mentioned in Remark \ref{R1}, 
the treatment of the nonlinearity $\bom\times \Del\bom$ is achieved by exploiting the cancellation structure in the bad interactions. See \eqref{ec7.625} and \eqref{ec7.65} for its details.
\subsection{Existence of solutions}
To show the existence of solution to \eqref{ec7.2}, we introduce a fourth-order regularization of the equation in the following form:
\begin{equation}\label{ec7.3}
\left\{
\begin{aligned}
\rd_t \bom &= - \be \Del^2 \bom 
+ \al \Del \bom - (\bh+\bom)\times \Del \bom 
+ N(\bom),\qquad (t,x)\in [0,\infty) \times\R^d,\\
\bom (0,x)&= \bom_0 (x),\qquad x\in \R^d,
\end{aligned}
\right.
\end{equation}
where $\be>0$ is a constant. 
Our construction proceeds by passing to the limit $\be\to 0$ for the solutions to \eqref{ec7.3}. 
For this sake, we first address the local well-posedness of \eqref{ec7.2}, and then derive a uniform bound of the solutions with respect to $\be\ll 1$. 
\begin{lem}\label{L7.13}
Let $s$ be an integer satisfying $s\ge 1$ for $d=1$ and $s\ge 2$ for $d=2,3$. 
For $\bom_0\in H^s(\R^d)$, there exists a unique mild solution $\bom \in C([0,T_{\max}^\be)\colon H^s(\R^d))\cap 
L^2_{\loc} ((0,T_{\max}^\be) \colon H^{s+2}(\R^d))
$, where $T_{\max}^\be$ is the maximal existence time in this class. 
Moreover, if $T_{\max}^\be <\infty$, then we have
\[
\lim_{t\to T_{\max}^\be} \nor{\bom(t)}{H^s(\R^d)} =\infty.
\]
\end{lem}
We remark that  
the solution $\bom$ in Lemma \ref{L7.13} is indeed in $C((0,T_{\max}^\be) \colon H^L(\R^d))$ for any $L\in\N$, which follows from repeated use of the smoothing effect. 
\begin{proof}
We rewrite \eqref{ec7.3} in the Duhamel formula:
\begin{equation}\label{ec7.4}
\bom(t) = e^{- t \be \Del^2} \bom_0 +
\int_0^t e^{- (t-\tau) \be \Del^2} \tilde{N} (\bom (\tau)) d\tau
\end{equation}
where 
\[
\tilde{N} (\bom ) = 
 \al \Del \bom - (\bh+\bom)\times \Del \bom 
+ N(\bom).
\]
Let us set the right hand side of \eqref{ec7.4} by $\Phi (\bom)$. 
We claim that for given $\bom_0\in H^s(\R^d)$, $\Phi$ is a contraction mapping on 
\begin{align*}
\tilde\boM=
\{ \bom\in C ([0,T]\colon H^s(\R^d) )\cap L^2(0,T\colon H^{s+2}(\R^d)) \\ 
|\ \nor{\bom}{L^\infty_T H^s_x
} + 
\nor{\Del \bom}{
L^2_TH^{s}_x
} \le M \}
\end{align*}
where 
\begin{equation}\label{ed7.0}
T= 
c (1+\nor{\bom_0}{H^s(\R^d)})^{-K}
,\qquad M= C \nor{\bom_0}{H^s(\R^d)}
\end{equation}
for some constants $c,C, K>0$. 
First, 
define the frequency cut-off operators 
$\Pi_L = \boF^{-1} 
\mathbbm{1}_{|\xi|\le 1} \boF$ 
and $\Pi_H = 1- \Pi_L$. 
Then, the following estimates for $e^{-t \be \Del^2}$ follow from 
the $L^\infty$-bounds for the Fourier multipliers, the energy estimate, and the interpolation inequalities: 
\begin{gather}\label{ed7.1}
\nor{ e^{-t\be\Del^2} f}{L^\infty_t L^2_{x}} +\nor{\Del e^{-t\be\Del^2} f}{L^2_{t,x}} \le C_\be \nor{f}{L^2_x},
\\
\label{ed7.2}
\nor{e^{-t \be\Del^2} \Pi_H f }{L^2_x} \le C_{\be,k} t^{-\frac k4}
\nor{|\nab|^{-k} \Pi_H f}{L^2_x},\quad t>0,\quad 0\le k\le 2,
\\
\label{ed7.3}
\nor{\Del \int_0^t e^{-(t-\tau)\be\Del^2} \Pi_H F (\tau) d\tau  
}{L^2_{t,x}} \le C_{\be,k} \nor{|\nab|^{-k} \Pi_H F}{L^{\frac{4}{4-k}}_tL^2_{x}}, \quad 0\le k\le 2.
\end{gather}
By \eqref{ed7.1}, we have
\[
\nor{e^{-t\be\Del^2} \bom_0}{L^\infty_T H^s_x}
+ \nor{\Del e^{-t\be\Del^2} \bom_0}{L^2_T H^s_x}
\le C\nor{\bom_0}{H^s_x}.
\]
Now we suppose $s\ge 2$. Then, using \eqref{ed7.1}--\eqref{ed7.3}, we estimate
\begin{align*}
&\nor{\int_0^t e^{-(t-\tau) \be\Del^2} \tilde{N} (\bom(\tau) )d\tau }{L^\infty_TH^s_x}\\
&\le 
C\nor{\Pi_L \int_0^t e^{-(t-\tau) \be\Del^2} \tilde{N} (\bom(\tau) )d\tau }{L^\infty_TL^2_x} \\
&\qquad + 
C\nor{\Pi_H \int_0^t e^{-(t-\tau) \be\Del^2} \tilde{N} (\bom(\tau) )d\tau }{L^\infty_TH^s_x}\\
&\le
C\int_0^T \nor{\Pi_L \tilde{N} (\bom(\tau))}{L^2_x} d\tau 
+ C\sup_{t\in [0,T]} \int_0^t (t-\tau)^{-\frac 12} \nor{\Pi_H \tilde{N}(\bom(\tau))}{H^{s-2}_x} d\tau \\
&\le C (T^{\frac 12} + T) \nor{\tilde{N} (\bom)}{L^\infty_T H^{s-2}_x},
\end{align*}
and
\begin{align*}
&\nor{\Del \int_0^t e^{-(t-\tau) \be\Del^2} \tilde{N} (\bom(\tau) )d\tau }{L^2_TH^s_x}\\
&\le 
C\nor{\Pi_L \int_0^t e^{-(t-\tau) \be\Del^2} \tilde{N} (\bom(\tau) )d\tau }{L^2_TL^2_x}\\
&\qquad
+ 
C\nor{\Pi_H \Del \int_0^t e^{-(t-\tau) \be\Del^2} \tilde{N} (\bom(\tau) )d\tau }{L^2_TH^s_x}\\
&\le
C T^{\frac 12} \nor{\Pi_L \tilde{N} (\bom)}{L^1_T L^2_x} + C \nor{\Pi_H \tilde{N}(\bom)}{L^2_TH^{s-2}_x} \\
&\le C (T^{\frac 12} + T^{\frac 32}) \nor{\tilde{N} (\bom)}{L^\infty_T H^{s-2}_x}.
\end{align*}
The nonlinear terms are estimated simply by the Leibniz rule and the Sobolev embedding, resulting in
\[
\nor{\tilde{N} (\bom)}{L^\infty_T H^{s-2}_x} \le C (\nor{\bom}{L^\infty_T H^s_x} + \nor{\bom}{L^\infty_T H^s_x}^{\tilde{K}} )
\]
for some integer $\tilde{K}\ge 1$. 
Thus, by setting $T$ and $M$ as in \eqref{ed7.0} 
with an appropriate choice of $c,C,K>0$, 
it follows that $\Phi$ maps $\tilde\boM$ into itself. When $d=1$ and $s=1$, we instead estimate the inhomogeneous term as
\begin{align*}
&\nor{\int_0^t e^{-(t-\tau) \be\Del^2} \tilde{N} (\bom(\tau) )d\tau }{L^\infty_TH^1_x}\\
&\le
C\int_0^T \nor{\Pi_L \tilde{N} (\bom(\tau))}{L^2_x} d\tau 
+ C\sup_{t\in [0,T]} \int_0^t (t-\tau)^{-\frac 14} \nor{\Pi_H \tilde{N}(\bom(\tau))}{L^2_x} d\tau \\
&\le C (T^{\frac 14} + T^{\frac 12}) \nor{\tilde{N} (\bom)}{L^2_T L^2_x},
\end{align*}
and
\begin{align*}
&\nor{\Del \int_0^t e^{-(t-\tau) \be\Del^2} \tilde{N} (\bom(\tau) )d\tau }{L^2_TH^1_x} \\
&\le 
CT^{\frac 12} \nor{\Pi_L \tilde{N}(\bom) }{L^1_TL^2_x} 
+ C \nor{\Pi_H \tilde{N}(\bom) }{L^{\frac{4}{3}}_TL^2_x}
\le C (T^{\frac 14} + T) \nor{\tilde{N} (\bom)}{L^2_T L^2_x}.
\end{align*}
By the Leibniz rule and the Sobolev embedding, we have
\[
\nor{\tilde{N} (\bom)}{L^2_T L^2_x} \le C
(1+T^{\frac 12})
(\nor{\bom}{L^\infty_T H^1_x} 
+ \nor{\Del \bom}{L^2_T H^1_x} 
+ \nor{\bom}{L^\infty_T H^1_x}^{\tilde{K}} 
+ \nor{\Del \bom}{L^2_T H^1_x}^{\tilde{K}})
\]
for some integer $\tilde{K}\ge 1$. 
Hence the same conclusion follows in this case as well. 
The difference estimate can be derived in a similar way, concluding that $\Phi$ is a contraction on $\tilde{\boM}$. 
The remaining part of the proof proceeds in a standard manner, and hence we omit the details.
\end{proof}
Next, we claim the following uniform bound in $\be$:
\begin{lem}\label{L7.14}
Let $s$ be an integer satisfying $s\ge 1$ when $d=1$ and $s\ge 2$ when $d=2,3$. 
Let $\bom(t) \in C([0,T_{\max}^\be) \colon H^s(\R^d))$ be a unique mild solution to \eqref{ec7.3}. 
Then, there exist $C= C(\nor{\bom_0}{H^s(\R^d)}), 
T=T(\nor{\bom_0}{H^s(\R^d)})>0$, independent of $\be$, such that 
$T_{\max}^\be \ge T$ for all $\be>0$, and
\[
\nor{\bom}{L^\infty (0,T\colon H^{s}_x)}
+ \nor{\nab \bom}{L^2 (0,T\colon H^{s}_x)} 
+ \be \nor{\Del \bom}{L^2 (0,T\colon H^{s}_x)} \le C.
\]
\end{lem}
\begin{proof}
We shall prove this via the energy method. 
Let $\gam \in \Z_+^d$ be any multi-index with $|\gam|\le s$. 
Taking $\rd^\gam$ on both sides of \eqref{ec7.3}, we obtain
\[
\rd_t \rd^\gam \bom = -\be \Del^2  \rd^\gam \bom 
+ \al \Del \rd^\gam \bom
- (\bh + \bom) \times \Del \rd^{\gam} \bom
+ \tilde{N}_\gam (\bom),
\]
where $\tilde{N}_\gam(\bom)$ is the remaining lower order terms:
\[
\tilde{N}_\gam (\bom) = 
-\sum_{
\substack{\gam^{(1)}+\gam^{(2)}=\gam\\ \gam^{(1)}\neq 0}}
(\rd^{\gam^{(1)}} \bh \times \Del \rd^{\gam^{(2)}} \bom
+ \rd^{\gam^{(1)}} \bom \times \Del \rd^{\gam^{(2)}} \bom
)
+ \rd^\gam N(\bom).
\]
By this, we have
\begin{align*}
\frac 12 \frac d{dt} \nor{\rd^\gam \bom}{L^2_x}^2 
&= \inp{\rd_t \rd^\gam \bom}{\rd^\gam \bom}_{L^2_x}\\
&=
-\be \nor{\Del \rd^\gam \bom}{L^2_x}^2 
-\al \nor{\nab \rd^\gam \bom}{L^2_x}^2 \\
&\qquad -\inp{(\bh + \bom ) \times \Del  \rd^\gam \bom}{\rd^\gam \bom}_{L^2_x} 
+ \inp{\tilde{N}_\gam (\bom)}{\rd^\gam \bom}_{L^2_x}.
\end{align*}
Integrating this over $[0,T]$ yields
\begin{equation}\label{ec7.5}
\begin{aligned}
& \frac 12 \nor{\rd^\gam \bom}{L^\infty_T L^2_x}^2 
+ \al \int_0^T \nor{\nab \rd^\gam \bom(t)}{L^2_x}^2 dt 
+ \be \int_0^T \nor{\Del \rd^\gam \bom(t)}{L^2_x}^2 dt\\
&= \frac 12 \nor{\rd^\gam \bom_0}{L^2_x}^2 + 
\int_0^T 
\left( -\inp{(\bh + \bom(t) ) \times \Del  \rd^\gam \bom (t)}{\rd^\gam \bom (t)}_{L^2_x} \right.\\
&\hspace{180pt}\left.+ \inp{\tilde{N}_\gam (\bom (t))}{\rd^\gam \bom (t)}_{L^2_x} \right) dt
\end{aligned}
\end{equation}
for $T\in (0, T_{\max}^\be)$. 
Now we estimate the right hand side of \eqref{ec7.5}. 
First, we claim that
\begin{equation}\label{e:17.1}
\begin{aligned}
&\int_0^T \inp{\tilde{N}_\gam (\bom (t))}{\rd^\gam \bom(t)}_{L^2_x}  dt 
\\
&\le 
C 
 (T^{\frac 12-\frac d4}+ T) 
\left(
 \nor{\bom}{L^\infty_T H^2_x}^2 
+ \nor{\nab \bom}{L^2_T H^{2}_x}^2
+ \nor{\bom}{L^\infty_T H^2_x}^K \right)
\end{aligned}
\end{equation}
for some $C>0$ and $K\ge 2$. 
If $s\ge 3$, \eqref{e:17.1} follows simply 
by the Leibniz rule and the Sobolev embedding $H^2(\R^d)\subset L^\infty(\R^d)$. 
If $s=1,2$, 
we need more subtlety to obtain \eqref{e:17.1} for $\gam$ with $|\gam|=s$. 
We shall observe the typical term $|\nab \bom|^2\bom$. 
Denoting $\rd$ to represent differentiation in any index, 
we have
\begin{align*}
&\int_0^T \left|\inp{\rd^2 (|\nab \bom|^2\bom)}{\rd^2 \bom}_{L^2_x}\right| dt \\
&\sim
\int_0^T \left|\inp{ (\rd^2 \nab \bom \cdot \nab \bom) \bom)}{\rd^2 \bom}_{L^2_x}\right| dt
+
\int_0^T \left|\inp{ (\rd \nab \bom \cdot \rd \nab \bom)\bom)}{\rd^2 \bom}_{L^2_x}\right| dt \\
&
\qquad +
\int_0^T \left|\inp{ (\rd \nab \bom \cdot \nab \bom) \rd\bom)}{\rd^2 \bom}_{L^2_x}\right| dt
+
\int_0^T \left|\inp{ |\nab \bom |^2 \rd^2 \bom)}{\rd^2 \bom}_{L^2_x}\right| dt \\
&
\lesssim
\int_0^T \nor{\rd^2 \nab \bom}{L^2_x} \nor{\nab \bom}{L^4_x} \nor{\bom}{L^\infty_x} \nor{\rd^2 \bom}{L^4_x} dt\\
&\qquad +
\int_0^T \nor{\rd \nab \bom}{L^4_x}^2 \nor{\bom}{L^\infty_x} \nor{\rd^2 \bom}{L^2_x}  dt \\
&
\qquad +
\int_0^T \nor{\rd \nab \bom}{L^4_x} \nor{\nab \bom}{L^4_x} \nor{\rd \bom}{L^4_x} \nor{\rd^2 \bom}{L^4_x} dt
+
\int_0^T \nor{\nab \bom}{L^4_x}^2 \nor{\rd^2 \bom}{L^4_x}^2  dt\\
&
\lesssim
\int_0^T \nor{\nab \bom}{H^2_x} \nor{\bom}{H^2_x}^2 
\left(\nor{\nab \rd^2 \bom}{L^2_x}^{\frac d4} \nor{\rd^2 \bom}{L^2_x}^{1-\frac d4}\right) dt\\
&\qquad +\int_0^T 
\nor{\bom}{H^2_x}^2
\left( \nor{\nab \rd^2 \bom}{L^2_x}^{\frac d4}  \nor{\rd^2 \bom}{L^2_x}^{1-\frac d4}\right)^2 
dt \\
&\lesssim 
T^{\frac 12 -\frac d8} \nor{\nab \bom}{L^2_TH^2_x}^{1+\frac d4} \nor{\bom}{L^\infty_T H^2_x}^{3-\frac d4} + T^{1-\frac d4}\nor{\nab \bom}{L^2_T H^2_x}^{\frac d2} \nor{\bom}{L^\infty_TH^2_x}^{4-\frac d2} \\
&\lesssim 
(T^{\frac 12 -\frac d8} + T^{1-\frac d4})(
\nor{\nab \bom}{L^2_TH^2_x}^{2} 
+ \nor{\bom}{L^\infty_TH^2_x}^{K})
\end{align*}
for some number $K\ge 2$, 
where we used the Gagliardo--Nirenberg inequality. 
When $d=1$ and $s=1$, 
using the Sobolev embedding $H^1(\R)\subset L^\infty(\R)$, we also have
\begin{align*}
&\int_0^T \left|\inp{\rd (|\rd \bom|^2\bom)}{\rd \bom}_{L^2_x}\right| dt \\
&\sim
\int_0^T \left|\inp{(\rd^2 \bom \cdot \rd \bom )\bom}{\rd \bom}_{L^2_x}\right| dt 
+
\int_0^T \left|\inp{|\rd \bom|^2 \rd\bom}{\rd \bom}_{L^2_x}\right| dt \\
&\lesssim 
\int_0^T 
\nor{\rd^2 \bom}{L^2_x} \nor{\rd \bom}{L^4_x} \nor{\bom}{L^\infty_x} \nor{\rd \bom}{L^4_x} dt
+
\int_0^T 
\nor{\rd \bom}{L^4_x}^4 dt\\
&\lesssim 
\int_0^T 
\nor{\rd^2 \bom}{L^2_x} \nor{\bom}{L^\infty_x}
\left(\nor{\rd^2 \bom}{L^2_x}^{\frac 14} \nor{\rd\bom}{L^2_x}^{\frac 34}\right)^{2}   dt\\
&\qquad  +
\int_0^T 
\left(\nor{\rd^2 \bom}{L^2_x}^{\frac 14} \nor{\rd\bom}{L^2_x}^{\frac 34}  \right)^4 dt
\\
&\lesssim 
T^{\frac 14} \nor{\rd \bom}{L^2_TH^1_x}^{\frac 32} \nor{\bom}{L^\infty_TH^1_x}^{\frac 52} 
+ 
T^{\frac 12} \nor{\rd \bom}{L^2_TH^1_x} \nor{\bom}{L^\infty_TH^1_x}^{3}\\
&\lesssim 
(T^{\frac 14} + T^{\frac 12})(
\nor{\rd \bom}{L^2_TH^1_x}^{2} 
+ \nor{\bom}{L^\infty_TH^1_x}^{K})
\end{align*}
for some $K\le 2$. 
The other terms in $\tilde{N}_\gam$ can be estimated similarly, yielding 
\eqref{e:17.1}. 
Next, we observe that the integration by part gives
\begin{equation}\label{ec7.625}
\begin{aligned}
\inp{(\bh + \bom ) \times \Del  \rd^\gam \bom}{\rd^\gam \bom}_{L^2_x} 
&=
\sum_{j=1}^d \inp{\rd_j [(\bh+\bom) \times \rd_j \rd^\gam \bom] }{ \rd^\gam \bom}_{L^2_x}
\\
&\qquad 
- \sum_{j=1}^d \inp{[\rd_j (\bh+\bom)] \times \rd_j \rd^\gam \bom }{ \rd^\gam \bom}_{L^2_x}\\
&=
- \sum_{j=1}^d \inp{[\rd_j (\bh+\bom)] \times \rd_j \rd^\gam \bom }{ \rd^\gam \bom}_{L^2_x}
,
\end{aligned}
\end{equation}
where we used the cancellation structure that
\begin{equation}\label{ec7.65}
\inp{(\bh +\bom) \times \rd_j \rd^\gam \bom}{ \rd_j \rd^\gam \bom}_{L^2_x}
= \int_{\R^d} (\bh +\bom) \times \rd_j \rd^\gam \bom \cdot \rd_j \rd^\gam \bom dx=0.
\end{equation}
Hence the second term in \eqref{ec7.5} can be estimated in the same way as in \eqref{e:17.1}. 
Therefore, we have
\begin{align*}
&\frac 12 \nor{\rd^\gam \bom}{L^\infty_T L^2_x}^2 
+ 
\al \nor{\nab \rd^\gam \bom}{L^2_T L^2_x}^2 
+\be \nor{\Del \rd^\gam \bom}{L^2_T L^2_x}^2 \\
&\le \frac 12 \nor{\rd^\gam \bom_0}{L^2_x}^2 + 
C (T^{\frac 12 -\frac d4} + T) \left(\nor{\bom}{L^\infty_T H^s_x}^2 
+ \nor{\nab \bom}{L^2_T H^s_x}^2
+ \nor{ \bom}{L^\infty_T H^s_x}^K \right).
\end{align*}
Thus, the conclusion follows if we set $T=T(\nor{\bom_0}{H^s})$, combined with 
a continuity argument.
\end{proof}
\begin{proof}[Proof of the existence part in Proposition \ref{P1}] 
For $\be>0$, 
let $\bom^\be$ be the solution to \eqref{ec7.3} in Lemma \ref{L7.13} with fixed initial data $\bom^\be (0,x)=\bom_0 (x)$. 
Then, Lemma \ref{L7.14} implies that 
\begin{equation}\label{ec7.7}
\sup_{0<\be<1} 
\left(
\nor{\bom^\be}{L^\infty_T H^{s}_x} +
\nor{\bom^\be}{L^2_T H^{s+1}_x} + \nor{\rd_t \bom^\be}{L^2_TL^2_x}
\right) <\infty
\end{equation}
with $T=T(\nor{\bom_0}{H^s})$ independent of $\be$. 
In particular, $\{\bom^\be\}_{\be>0}$ is bounded in $H^1((0,T)\times \R^d)$. 
Therefore, there exists a subsequence $\{\bom^{\be_n} \}_{n\in\N}$ and $\bom\colon (0,T)\times \R^d\to \R^3$ such that
\begin{align*}
\bom^{\be_n} \wto \bom  
\quad \text{ in } L^2_TH^{s+1}(\R^d),&& 
\text{and}&&
\bom^{\be_n} \wto \bom \quad \text{ in }
H^1((0,T)\times B_R(0)).
\end{align*}
Moreover, by the Rellich--Kondrachov theorem, we also have
\[
\bom^{\be_n} \to \bom \quad \text{strongly in } L^2((0,T)\times \R^d) \text{ for any } R>0.
\]
Then, applying the same argument in \cite{MR2304153}, 
we can see that 
$\bom^{\be_n}$ in fact converges to $\bom$ strongly in $L^\infty(0,T\colon B_R(0))$ for any $R>0$. 
In particular, it follows that $\bom\in C([0,T]\colon B_R(0))$ for each $R>0$, and that $\bom(0)= \bom_0$.\par
Next, we show that $\bom$ weakly solves \eqref{E1.7}. 
Let $\bpsi\in C_0^\infty ((0,T)\times \R^d\colon \R^3)$. 
Then for each $n$, we have
\begin{equation}\label{ec7.75}
\begin{aligned}
&-\int_{(0,T)\times \R^d} \bom^{\be_n} \cdot \rd_t \bpsi  dt dx \\
&=
-\be_n \int_{(0,T)\times \R^d} \bom^{\be_n} \cdot\Del^2 \bpsi  dt dx
+\al \int_{(0,T)\times \R^d} \bom^{\be_n} \cdot \Del \bpsi  dt dx \\
&\qquad 
 +\int_{(0,T)\times \R^d} \left( -(\bh +\bom^{\be_n})\times \Del \bom^{\be_n} 
+
N(\bom^{\be_n})\right)  \cdot \bpsi dt dx.
\end{aligned}
\end{equation}
Here, the Leibniz rule and the Sobolev inequality yield
\begin{align*}
&\nor{(\bh+\bom^{\be_n}) \times \Del \bom^{\be_n}}{L^2_T H^{s-1}_x} +
\nor{N(\bom^{\be_n})}{L^2_T H^{s-1}_x} \\
&\le C = C \left(s,
\nor{\bom^{\be_n}}{L^\infty_T H^s_x}, \nor{\bom^{\be_n}}{L^2_T H^{s+1}_x }
\right),
\end{align*}
which is bounded in $n\in\N$. 
Hence, there is a subsequence of $\{\be_n\}_{n=1}^\infty$, where we again write it in the same notation, such that 
$(\bh+\bom^{\be_n}) \times \Del \bom^{\be_n} \wto (\bh+\bom) \times \Del \bom$, 
$N(\bom^{\be_n}) \wto N(\bom)$ in $L^2_T H^{s-1}_x$ as $n\to\infty$. 
Therefore, taking $n\to\infty$ for \eqref{ec7.75}, we have
\begin{align*}
&-\int_{(0,T)\times \R^d} \bom \cdot \rd_t \bpsi dt dx \\
&=
\al \int_{(0,T)\times \R^d} \bom \cdot \Del \bpsi dt dx
+\int_{(0,T)\times \R^d} \left( -(\bh + \bom)\times \Del \bom + N(\bom) \right) \cdot \bpsi dt dx,
\end{align*}
which implies that $\bom$ is a weak solution to \eqref{E1.7}.\par
We finally show that $\bom$ has the desired regularity. Since $\bom_0\in H^s(\R^d)$, 
$\bom \in L^2_TH^{s+1}_x$, and $N(\bom)\in L^2_TH^{s-1}_x$, 
if we apply 
Corollary 4.1.9 in \cite{MR1691574} with $X=L^2(\R^d)$ and $A=\al \Del$, 
then it follows that $\bom$ is actually a mild solution to \eqref{ec7.3}:
\[
\bom (t) = e^{\al t\Del} \bom_0 + \int_0^t e^{\al (t-\tau) \Del} N(\bom)(\tau) d\tau .
\]
This particularly implies that $\bom\in C([0,T]\colon  H^s(\R^d))\cap L^2(0,T\colon H^{s+1}(\R^d))$ as desired.
\end{proof}
\subsection{Conclusive step}\label{SA2}
We next address the uniqueness and the continuous dependence on initial data. 
Let $T>0$, and 
let $\bom^1$, $\bom^2\in C([0,T]\colon H^{s-1}(\R^d))$ be solutions to \eqref{ec7.2} with initial data $\bom^1_0$, $\bom^2_0$, respectively. Set $\del \bom = \bom^1 -\bom^2$ and $\del \bom_0 = \bom^1_0 - \bom^2_0$. 
Then, $\del \bom$ satisfies
\[
\left\{
\begin{aligned}
\rd_t \del \bom &= \al \Del \del \bom - (\bh + \bom^1) \times \Del \del \bom
+ \del N,\quad (t,x)\in [0,\infty) \times \R^d \\
\del \bom (0,x) &= \del \bom_0 (x),\quad x\in \R^d,  
\end{aligned}
\right.
\]
where
\[
\del N := 
-\del \bom \times \Del \bom^2 + N(\bom^1) - N(\bom^2).
\]
For $\tilde{T} \in (0,T)$, 
we apply the same energy method as in the proof of Lemma \ref{L7.14}, yielding
\begin{equation}\label{ec7.8}
\begin{aligned}
&\frac 12 \nor{\del \bom}{L^\infty_{\tilde{T}} H^{s-1}_x}^2 
+ \al \nor{\nab \del \bom}{L^2_{\tilde{T}} H^{s-1}_x}^2 \\
&\le \frac 12 \nor{\del \bom_0}{H^{s-1}_x}^2 
+ C (\tilde{T}^{\frac 12- \frac d4} + \tilde{T}) (1+ M^K) 
\left(
\nor{\del \bom}{L^\infty_{\tilde{T}} H^{s-1}_x}^2 
+\nor{\nab \del \bom}{L^2_{\tilde{T}} H^{s-1}_x}^2 
\right)
\end{aligned}
\end{equation}
for some $C>0$ and $K> 0$, where $M= \nor{\bom^1}{L^\infty_{\tilde{T}}H^{s}_x} + \nor{\bom^2}{L^\infty_{\tilde{T}}H^s_x}
+\nor{\nab \bom^1}{L^2_{\tilde{T}}H^{s}_x} + \nor{\nab \bom^2}{L^2_{\tilde{T}}H^{s}_x}
$. 
Therefore, if we take $\tilde{T}$ sufficiently small accordingly to $M$, we have
\begin{equation}\label{ec7.9}
\nor{\del \bom}{L^\infty_{\tilde{T}} H^{s-1}_x} \le  C \nor{\del \bom_0}{H^{s-1}_x}.
\end{equation}
Repeating the above argument extends the interval into $[0,T]$, 
implying the desired conclusion.\par
Finally, suppose that $|\bh(x) + \bom_0(x)|= 1$ for all $x\in\R^d$. Set 
\[
F(t,x) = e^{-2\al t}\left( |\bh (x)+\bom(t,x)|^2 -1\right) = 
e^{-2\al t} \left(2 \bh(x)\cdot \bom(t,x) + |\bom(t,x)|^2\right).
\] 
Then 
by direct computation, 
\eqref{ec7.2} implies that $F$ satisfies
\[
\rd_t F = \al \Del F,\qquad F(0,x)=0.
\]
Since $F\in C_tL^2_x$, the 
uniqueness for the heat equation implies that $F\equiv 0$, and hence 
$|\bh(x)+\bom(t,x)|=1$ for all $(t,x)\in [0, T_{\max})\times \R^d$. 
Therefore, the proof is complete.
\section{Proof of \eqref{E2.3}}\label{SB}
In this section, we give a proof of \eqref{E2.3}. We first recall the setting and the claim. 
Let $u\colon \R^d\to\C$ be a function with $|u|\le \frac 12$ on $\R^d$, 
and let $z$, $\bom$ be the associated function in \eqref{E2.2}, \eqref{E2.1}, respectively. Then we show the following:
\begin{lem}
For $s\in \Z_{\ge 0}$ and 
$p\in [\frac d2,\infty]$ with $p>1$, 
there exists $C=C(s,p)>0$ such that
\begin{align}
\nor{\bom}{W^{s,p}(\R^d)} 
&\le C
\left(\nor{u}{W^{s,p}(\R^d)} + 
\nor{u}{W^{s,p}(\R^d)}^{\max\{s,1\}}\right), \label{EB1} \\
\nor{u}{W^{s,p}(\R^d)} 
&\le C  \nor{\bom}{W^{s,p}(\R^d)}.\label{EB2}
\end{align}
\end{lem}
\begin{proof}
We begin by proving \eqref{EB2}. It suffices to show
\begin{equation}\label{EB3}
\nor{\bm{F}\cdot \bJ_1}{W^{s,p}}
+\nor{\bm{F}\cdot \bJ_2}{W^{s,p}}
+\nor{\bm{F}\cdot \bh}{W^{s,p}}
\lesssim \nor{\bm{F}}{W^{s,p}}
\end{equation}
for $\bm{F}\in W^{s,p}(\R^d\colon \R^3)$. 
We proceed by induction on $s$. The case $s=0$ is immediate. If \eqref{EB3} is true for some $s\in\Z_{\ge 0}$, 
we have
\begin{align*}
\nor{\bm{F}\cdot \bJ_1}{W^{s+1,p}}
= \nor{\nab (\bm{F}\cdot \bJ_1)}{W^{s,p}}
&\le \sum_{j=1}^d \nor{(\rd_j \bm{F})\cdot \bJ_1}{W^{s,p}}
+\nor{ \bm{F}\cdot (\rd_1 \bJ_1)}{W^{s,p}}\\
&= \sum_{j=1}^d \nor{(\rd_j \bm{F})\cdot \bJ_1}{W^{s,p}}
+\nor{ \bm{F}\cdot \bh}{W^{s,p}}\\
&\lesssim \sum_{j=1}^d 
\nor{\rd_j\bm{F}}{W^{s,p}} + \nor{\bm{F}}{W^{s,p}}
\lesssim \nor{\bm{F}}{W^{s+1,p}}.
\end{align*}
$\nor{\bm{F}\cdot \bJ_2}{W^{s+1,p}}$, $\nor{\bm{F}\cdot \bh}{W^{s+1,p}}$ can be estimated in the similar way, leading to \eqref{EB3} for $s+1$. Hence \eqref{EB3} holds true for all $s\in\Z_{\ge 0}$.\par
Next, we prove \eqref{EB1}. 
By definition, we have
\[
\nor{\bom}{W^{s,p}} \le \nor{u_1\bJ_1}{W^{s,p}} + \nor{u_2\bJ_2}{W^{s,p}}
+\nor{z\bh}{W^{s,p}}.
\]
For $\gam\in (\Z_{\ge 0})^d$ with $|\gam|\le s$, the Leibniz rule yields
\begin{align*}
&\nor{\rd^\gam (u_1 \bJ_1)}{L^p} + 
\nor{\rd^\gam (u_2 \bJ_2)}{L^p} \\
&\lesssim \sum_{\gam_1+\gam_2=\gam} \nor{\rd^{\gam_1} u_1 \rd^{\gam_2}\bJ_1}{L^p}
+ \nor{\rd^{\gam} u_2 }{L^p}
\lesssim \nor{u}{W^{s,p}}.
\end{align*}
Now we claim 
\begin{equation}\label{EB4}
\nor{\rd^\gam z}{L^p}\lesssim 
\nor{u}{W^{s,p}} + \nor{u}{W^{s,p}}^{\max\{s,1\}}
\end{equation}
for $\gam\in (\Z_{\ge 0})^d$ with $|\gam|= s\ge 0$. 
First, the condition $|u|\le \frac 12$ and \eqref{E2.25} imply
\begin{align*}
\nor{z}{L^p} \lesssim \nor{u^2}{L^p} 
\lesssim \nor{u}{L^p},&&
\nor{\nab z}{L^p} 
\lesssim \nor{u \nab u}{L^p} 
\lesssim \nor{u}{L^p}
\end{align*}
and hence \eqref{EB4} holds for $s=0,1$. 
Here, note that
\begin{align}\label{e2:3}
\rd_j \rd_k z 
= 
\frac 12(1-|u|^2)^{-\frac 32} 
\Re (\ovl{u}\rd_j u ) 
- (1-|u|^2)^{-\frac 12} 
\Re (\ovl{u}\rd_j\rd_k u + \ovl{\rd_j u}\rd_k u)
\end{align}
for $j,k=1,...,d$. 
Since our assumption implies $\frac{1}{p}\le \frac 1{2p} + \frac{1}{d}$, the Sobolev embedding yields
\[
\nor{\nab^2 z}{L^p} \lesssim 
\nor{\nab^2 u}{L^p} +\nor{\nab u}{L^{2p}}^2 
\lesssim \nor{u}{W^{2,p}} + \nor{\nab u}{W^{1,p}}^2, 
\]
which gives \eqref{EB4} for $s=2$. 
Next we consider the case $s\ge 3$. 
By induction, one easily sees that $\rd^\gam z$ can be written as a linear combination of terms of the form
\begin{equation}\label{EB5}
(1-|u|^2)^{-\frac 12-k_0} u^{k_1} \ovl{u}^{k_2} \rd^{\gam_1} u \cdots\rd^{\gam_{l'}} u \rd^{\gam_{l'+1}} \ovl{u} \cdots \rd^{\gam_l} \ovl{u}
\end{equation}
with $0\le l'\le l \le s$, $k_0,k_1,k_2\in \Z_{\ge 0}$, $\gam_j\in (\Z_{\ge0})^d \setminus \{0\}$ satisfying $\gam_1+\cdots +\gam_l = \gam$. 
If $p=\infty$, then 
\begin{align*}
&\nor{(1-|u|^2)^{-\frac 12-k_0} u^{k_1} \ovl{u}^{k_2} \rd^{\gam_1} u \cdots\rd^{\gam_{l'}} u \rd^{\gam_{l'+1}} \ovl{u} \cdots \rd^{\gam_l} \ovl{u}}{L^\infty}\\
&\lesssim \nor{\rd^{\gam_1} u}{L^{\infty}} \cdots \nor{\rd^{\gam_l} u}{L^{\infty}} \lesssim \nor{u}{W^{s,\infty}}^l 
\lesssim \nor{u}{W^{s,\infty}} + 
\nor{u}{W^{s,\infty}}^s 
,
\end{align*}
which implies \eqref{EB4}. 
Now suppose that $p<\infty$. 
For simplicity, we assume $|\gam_1|\ge ... \ge |\gam_l|$. 
If $\gam_1=\gam$, then $l=1$ and hence
\begin{align*}
&\nor{(1-|u|^2)^{-\frac 12-k_0} u^{k_1} \ovl{u}^{k_2} \rd^{\gam_1} u \cdots\rd^{\gam_{l'}} u \rd^{\gam_{l'+1}} \ovl{u} \cdots \rd^{\gam_l} \ovl{u}}{L^p}\lesssim \nor{\rd^{\gam_1} u}{L^{p}}  \lesssim \nor{u}{W^{s,p}}.
\end{align*}
If $|\gam_1|\le s-1$, then the assumption $s\ge 3$ implies $|\gam_2|\le s-2$. 
Noting that $W^{2,p}(\R^d)\subset L^q(\R^d)$ holds for any $p \le q<\infty$, we have
\begin{align*}
&\nor{(1-|u|^2)^{-\frac 12-k_0} u^{k_1} \ovl{u}^{k_2} \rd^{\gam_1} u \cdots\rd^{\gam_{l'}} u \rd^{\gam_{l'+1}} \ovl{u} \cdots \rd^{\gam_l} \ovl{u}}{L^p}\\
&\lesssim \nor{\rd^{\gam_1} u}{L^{2p}} 
\nor{\rd^{\gam_2} u}{L^{2p(l-1)}}\cdot\cdots\cdot \nor{\rd^{\gam_l} u}{L^{2p(l-1)}} \\
& \lesssim \nor{\rd^{\gam_1} u}{W^{1,p}} 
\nor{\rd^{\gam_2} u}{W^{2,p}}\cdot\cdots\cdot \nor{\rd^{\gam_l} u}{W^{2,p}}\lesssim \nor{u}{W^{s,p}}^l 
\lesssim \nor{u}{W^{s,p}} +  \nor{u}{W^{s,p}}^s,
\end{align*}
which shows \eqref{EB4}. 
Consequently, the Leibniz rule yields
\[
\nor{\rd^\gam (z\bh)}{L^p} 
\lesssim \sum_{\gam_1+\gam_2=\gam} \nor{\rd^{\gam_1} z \rd^{\gam_2} \bh}{L^p}
\lesssim \sum_{0\le\gam_1\le\gam}\nor{\rd^{\gam_1}z}{L^p}
\lesssim \nor{u}{W^{s,p}} + \nor{u}{W^{s,p}}^{\max\{s,1\}},
\]
concluding \eqref{EB1}.
\end{proof}
\textbf{Acknowledgement}
The author is supported by JSPS KAKENHI Grant Number JP23KJ1416. 
\bibliography{helix}
\end{document}